\documentclass{amsart}
\usepackage{amsmath,amsfonts,amsthm,amssymb,latexsym}
\usepackage{graphicx,float,wrapfig,epsf, caption, subcaption}
\usepackage{epsfig}
\usepackage{cyr}
\usepackage{psfrag}
\usepackage{tikz}
\usetikzlibrary{arrows,positioning,calc}

\newtheorem{thm}{Theorem}[section]
\newtheorem{cor}[thm]{Corollary}
\newtheorem{lem}[thm]{Lemma}

\newtheorem{prop}[thm]{Proposition}

\newtheorem{Example}[thm]{Example}

\def\bu{\bullet}

\def\zz{\mathbb Z}

\def\rr{\mathbb R}

\def\De{\Delta}

\def\la{\lambda}

\def\si{\sigma}

\def\al{\alpha}

\def\om{\omega}

\def\ve{\varepsilon}

\def\CS{\mathcal S}

\def\cd{\mathcal D}

\def\CT{\mathcal T}

\def\<{\langle}
\def\>{\rangle}

\def\.{\hskip.04cm}
\def\ts{\hskip.02cm}

\def\0{{\mathbf 0}}

\def\fp{{\text {\rm fp} } }
\def\rmax{{\text {\rm rmax} } }
\def\ldr{{\text {\rm ldr} } }
\def\afp{{\text {\rm afp} } }
\def\lis{{\text {\rm lis} } }

\def\eE{\mathbf{E}}

\def\last{{\text{\rm last}}}
\def\first{{\text{\rm first}}}

\def\rank{{\text{\rm rank}}}

\def\bP{{\mathbf{P}}}

\tikzset{
     nicearrow/.style={
     	->,
	thick,
	shorten <=2pt,
	shorten >=2pt,}
}

\pgfdeclareverticalshading{myshadingE}{80bp}
{color(-10bp)=(white); color(80bp)=(red)}

\begin{document}

\title[Shape of pattern-avoiding permutations]
{The shape of random pattern-avoiding permutations}

\author[Sam~Miner]{ \ Sam~Miner$^\star$}
\author[Igor~Pak]{ \ Igor~Pak$^\star$}

\thanks{\thinspace ${\hspace{-.45ex}}^\star$
Department of Mathematics,
UCLA, Los Angeles, CA, 90095.
\hskip.06cm
Email:
\hskip.02cm
\texttt{\{samminer,pak\}@math.ucla.edu}}
\maketitle

\date{\today}

\vskip1.3cm

\begin{abstract}
We initiate the study of limit shapes for random permutations avoiding
a given pattern.  Specifically, for patterns of length~3, we obtain delicate
results on the asymptotics of distributions of positions of numbers in
the permutations.  We view the permutations as \ts $0$-$1$~matrices to
describe the resulting asymptotics geometrically.  We then apply our
results to obtain a number of results on distributions
of permutation statistics.
\end{abstract}

\vskip.9cm


\section*{Introduction}

\noindent
The \emph{Catalan numbers} is one of the most celebrated integer sequences,
so much that it is hard to overstate their importance and applicability.
In the words of Thomas Koshy, ``Catalan numbers are [..] fascinating.
Like the North Star in the evening sky, they are a beautiful
and bright light in the mathematical heavens.''  Richard Stanley
called them ``the most special'' and his ``favorite number sequence''~\cite{Kim}.
To quote Martin Gardner, ``they have the delightful propensity for
popping up unexpectedly, particularly in combinatorial problems"~\cite{G}.
In fact, Henry Gould's bibliography~\cite{Gould} lists over 450 papers on the subject,
with many more in recent years.

Just as there are many combinatorial interpretations of Catalan numbers~\cite[Exc.~6.19]{Sta}
(see also~\cite{Pak,Slo,S2}), there are numerous results on statistics of various such
interpretations (see e.g.~\cite{B1,Sta}), as well as their probabilistic and
asymptotic behavior (see~\cite{Drm,FS}).  The latter results usually come in
two flavors.  First, one can study the probability distribution of statistics,
such as the expectation, the standard deviation and higher moments.
The approach we favor is to define the shape of a large random object,
which can be then analyzed by analytic means (see e.g.~\cite{A2,Ver,VK}).
Such objects then contain information about a number of statistics,
under one roof.

\smallskip

In this paper we study the set $\CS_n(\pi)$ of permutations
$\si\in S_n$ \emph{avoiding a pattern}~$\pi$.  This study was initiated by
Percy MacMahon and Don Knuth, who showed that the size of $\CS_n(\pi)$ is
the Catalan number~$C_n$, for all permutations $\pi\in S_3$~\cite{Knu,M}.
These results opened a way to a large area of study, with numerous
connections to other fields and applications~\cite{Kit} (see also
Subsection~\ref{ss:fin-pat}).

We concentrate on two classical patterns, the \textbf{123}- and
\textbf{132}-avoiding permutations.  Natural symmetries imply that
other patterns in~$S_3$ are equinumerous with these two patterns.  We view permutations
as \ts $0$-$1$~matrices, which we average, scale to fit a unit square,
and study the asymptotic behavior of the resulting family of distributions.
Perhaps surprisingly, behavior of these two patterns is similar on a small scale
(linear in~$n$), with random permutations approximating the reverse identity
permutation $(n,n-1,\ldots,1)$.
However, on a larger scale (roughly, on the order $n^\al$ away from the diagonal),
the asymptotics of shapes of random permutations in $\CS_n(123)$ and~$\CS_n(132)$,
are substantially different.  This explains, perhaps, why there are at least
nine different bijections between two sets, all with different properties,
and none truly ``ultimate'' or ``from the book'' (see Subsection~\ref{ss:fin-geom}).

Our results are rather technical and contain detailed information
about the random pattern avoiding permutations, on both the small and large scale.
We exhibit several regimes (or ``phases''), where the asymptotics are unchanged,
and painstakingly compute the precise limits, both inside these regimes and at
the phase transitions.  Qualitatively, for \textbf{123}-avoiding permutations,
our results are somewhat unsurprising, and can be explained by the limit
shape results on the \emph{brownian excursion} (see Subsection~\ref{ss:fin-be});
still, our results go far beyond what was known.  However, for the
\textbf{132}-avoiding permutations, our results are extremely unusual,
and have yet to be explained even on a qualitative level
(see Subsection~\ref{ss:fin-proc}).

\smallskip

The rest of the paper is structured as follows.  In the next section
we first present examples and calculations which then illustrate
the ``big picture'' of our results.  In Section~\ref{s:def} we
give formal definitions of our matrix distributions and state
basic observations on their behavior.  We state the main results
in Section~\ref{s:main}, in a series of six theorems of increasing
complexity, for the shape of random permutations in $\CS_n(123)$
and~$\CS_n(132)$, three for each.  Sections~\ref{s:123} and~\ref{s:132}
contain proofs of the theorems.  In the next three sections
(sections~\ref{s:exp},~\ref{s:fp} and~\ref{s:lis}), we give a long
series of corollaries, deriving the distributions for the positions
of~$1$ and~$n$, the number and location of fixed points, and the generalized rank.
We conclude with final remarks and open problems (Section~\ref{s:fin}).

\bigskip

\section{The big picture} \label{s:big}

In this section we attempt to give a casual description of our results,
which basically makes this the second, technical part of the
introduction.\footnote{Here and always when in doubt,
we follow Gian-Carlo Rota's
advice on how to write an Introduction~\cite{Rota}.}

\subsection{The setup} \label{s:big-setup}
Let $P_n(j,k)$ and $Q_n(j,k)$ be the number of \textbf{123}- and
\textbf{132}-avoiding permutations, respectively, of size~$n$,
that have $j$ in the $k$-th position.  These are the main quantities
which we study in this paper.

There are two ways to think of $P_n(\cdot,\cdot)$ and $Q_n(\cdot,\cdot)$.
First, we can think of these as families of probability distributions
$$
\frac{1}{C_n}\ts P_n(j,\cdot) \,, \quad \frac{1}{C_n}\ts P_n(\cdot,k)
\,, \quad
\frac{1}{C_n}\ts Q_n(j,\cdot)\,, \quad \, \text{and} \quad
\frac{1}{C_n}\ts Q_n(\cdot,k)\..
$$
In this setting, we find the asymptotic behavior of these distributions,
where they are concentrated and the tail asymptotics; we also find exactly
how they depend on parameters $j$ and~$k$.

Alternatively, one can think of $P_n(\cdot,\cdot)$ and $Q_n(\cdot,\cdot)$
as single objects, which we can view as a bistochastic matrices:
$$
\textrm{P}_n \. = \. \frac{1}{C_n} \. \sum_{\si \in \CS(123)} \. M(\si)\ts, \qquad
\textrm{Q}_n \. = \. \frac{1}{C_n} \. \sum_{\si \in \CS(132)} \. M(\si)\ts,
$$
where $M(\si)$ is a permutation matrix of $\si \in S_n$, defined so that
$$
M(\si)_{jk}:=\begin{cases}1 & \si(j)=k\\0 & \si(j) \neq k.\end{cases}
$$
This approach is equivalent to the first, but more conceptual and visually transparent,
since both $\textrm{P}_n$ and~$\textrm{Q}_n$ have nice geometric asymptotic behavior when
$n\to \infty$.  See Subsection~\ref{ss:fin-hist} for more on this difference.

Let us present the ``big picture'' of our results.  Roughly, we show that both
matrices~$\textrm{P}_n$ and~$\textrm{Q}_n$ are very small for $(j,k)$ sufficiently far away
from the \emph{anti-diagonal}
$$\De\.=\.\{(j,k) \.\mid\. j+k=n+1\}\ts,$$
and from the lower right corner~$(n,n)$ in the case
of~$\textrm{Q}_n$.  However, already on the next level of detail there are large
differences: $\textrm{P}_n$ is exponentially small away from the anti-diagonal,
while $\textrm{Q}_n$ is exponentially small only above~$\De$, and decreases at a
rate $\Theta(n^{-3/2})$ on squares below~$\De$.

At the next level of detail, we look inside the ``phase transition'', that is
what happens when $(j,k)$ are near~$\De$.  It turns out, matrix~$\textrm{P}_n$
maximizes at distance $\Theta(\sqrt{n})$ away from~$\De$, where the values scale as
$\Theta(n^{-1/2})$, i.e.~much greater than average~$1/n$.  On the other hand,
on the anti-diagonal~$\De$, the values of $P_n$ scale as $\Theta(n^{-3/2})$,
i.e.~below the average.  A similar, but much more complicated phenomenon
happens for~$\textrm{Q}_n$.  Here the ``phase transition'' splits into several phases,
with different asymptotics for rate of decrease, depending on how the distance
from $(j,k)$ to~$\De$ relates to $\Theta(\sqrt{n})$ and $\Theta(n^{3/8})$
(see Section~\ref{s:main}).

At an even greater level of detail, we obtain exact asymptotic constants on the
asymptotic behavior of~$\textrm{P}_n$ and~$\textrm{Q}_n$, not just the rate of decrease.  For example,
we consider $\textrm{P}_n$ at distance $\Theta(n^{1/2+\ve})$ from~$\De$, and show that $\textrm{P}_n$ is
exponentially small for all $\ve>0$.  We also show that below~$\De$,
the constant term implied by the $\Theta$~notation in the rate $\Theta(n^{-3/2})$
of decrease of~$Q_n$, is itself decreasing until ``midpoint'' distance $n/2$
from~$\De$, and is increasing beyond that, in a symmetric fashion.

Unfortunately, the level of technical detail of our theorems is a bit overwhelming
to give a casual description; they are formally presented in Section~\ref{s:main}, and
proved in sections~\ref{s:123} and~\ref{s:132}.  The proofs rely on explicit formulas
for $P_n(j,k)$ and~$Q_n(j,k)$ which we give in lemmas~\ref{explicitP} and~\ref{explicitQ}.
These are proved by direct combinatorial arguments.  From that point on, the proofs of
the asymptotic behavior of $\textrm{P}_n$ and~$\textrm{Q}_n$ are analytic and use no tools
beyond Stirling's formula and the Analysis of Special Functions.

\medskip

\subsection{Numerical examples} \label{ss:big-num}
First, in Figures~\ref{P250} and~\ref{Q250}
we compute the graphs of $\textrm{P}_{250}$ and~$\textrm{Q}_{250}$
(see the Appendix).   Informally, we name the diagonal mid-section of the
graph of $\textrm{P}_{250}$ the \emph{canoe}; this is the section of the graph
where the values are the largest.  Similarly, we use the \emph{wall} for
the corresponding mid-section of $\textrm{Q}_{250}$ minus the corner spike.
The close-up views of the canoe and the wall are given in
Figures~\ref{P250close} and~\ref{Q250close}, respectively.
Note that both graphs here are quite smooth, since $n=250$ is
large enough to see the the limit shape, with
$C_{250} \approx 4.65\times 10^{146}$, and every pixelated
value is computed exactly rather than approximated.

Observe that the canoe is symmetric across both the main and the anti-diagonal,
and contain the high spikes in the corners of the canoe, both of which
reach~$1/4$.  Similarly, the wall is symmetric with respect to the main diagonal,
and has three spikes which reach~$1/4$.  These results are straightforward
and proved in the next section.

To see that the canoe is very thin, we compare graphs of the diagonal
sections $P_n(k,k)/C_n$ for $n=62,125,250$ and~$500$, as $k$ varies
from~90 to~160 (see Figures~\ref{Qkk62} to~\ref{Q500kk} in the Appendix).
Observe that as $n$ increases, the height
of the canoes decreases, and so does the width and ``bottom''.
As we mentioned earlier, these three scale as $\Theta(n^{-1/2})$,
$\Theta(n^{-1/2})$, and $\Theta(n^{-3/2})$, respectively.  Note also
the sharp transition to a near flat part outside of the canoe; this is
explained by an exponential decrease mentioned earlier.  The exact
statements of these results are given in Section~\ref{s:main}.

Now, it is perhaps not clear from Figure~\ref{P250close} that the
wall bends to the left.  To see this clearly, we overlap two graphs in
Figure~\ref{PQ250}.  Note that the peak of $P_{250}(k,k)$ is roughly
in the same place of $Q_{250}(k,k)$, i.e. well~to the left of the midpoint
at~125.  The exact computations show that the maxima occur at 118 and at~119,
respectively.  Note also that $Q_{250}(k,k)$ has a sharp phase transition
on the left, with an exponential decay, but only a polynomial decay
on the right.

\begin{figure}[hbt]
\centering
\includegraphics[width=60mm]{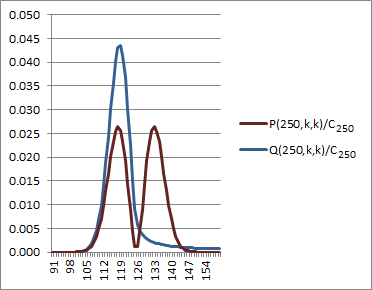}
\caption{Comparison of $P_{250}(k,k)/C_{250}$ and $Q_{250}(k,k)/C_{250}$.}
\label{PQ250}
\end{figure}

\medskip

\subsection{Applications} \label{ss:big-app}
We mention only one statistic which was heavily studied in previous years, and
which has a nice geometric meaning.  Permutation~$\si$ is said to have a
\emph{fixed point} at~$k$ if $\si(k)=k$.  Denote by $\fp(\si)$ the number
of fixed points in~$\si$.

For random permutations $\si\in S_n$, the distribution of $\fp$ is a classical
problem; for example $\eE[\fp]=1$ for all~$n$.  In an interesting paper~\cite{RSZ},
the authors prove that the distribution of~$\fp$ on $\CS(321)$ and on $\CS(132)$
coincide (see also~\cite{E2,EP}).  Curiously, Elizalde used the generating function
technique to prove that $\eE[\fp]=1$ in both cases, for all~$n$, see~\cite{E1}.
He also finds closed form g.f.~formulas for the remaining two patterns
(up to symmetry).

Now, the graphs of $P_n(k,k)/C_n$ and $Q_n(k,k)/C_n$ discussed above, give the
expectations that~$k$ is a fixed point in a pattern avoiding permutation.
In other words, fixed points of random permutations in $\CS_n(123)$ and
$\CS_n(132)$ are concentrated under the canoe and under the wall, respectively.
Indeed, our results immediately imply that w.h.p.~they lie near~$n/2$ in both cases.
For random permutations in $\CS_n(321)$ and $\CS_n(231)$, the fixed points
lie in the ends of the canoe and near the corners of the wall, respectively.
In Section~\ref{s:fp}, we qualify all these statements and as a show of force
obtain sharp asymptotics for $\eE[\fp]$ in all cases, both known and new.

\bigskip

\section{Definitions, notations and basic observations}\label{s:def}

\subsection{Asymptotics}
Throughout this paper we use $f(n) \sim g(n)$ to denote
\[\lim_{n \to \infty}\frac{f(n)}{g(n)} = 1\,.\]
We use $f(n) = O(g(n))$ to mean that there exists a constant $M$
and an integer $N$ such that
\[\vert f(n) \vert \leq \,M\, \vert g(n) \vert \quad \text{ for all }\,\, n > N.\]
Also, $f(n) = \Theta(g(n))$ denotes that $f(n) = O(g(n))$ and $g(n) = O(f(n))$.
Similarly, $f(n) = o(g(n))$ is defined by
\[\lim_{n \to \infty}\frac{f(n)}{g(n)}=0\,.\]
Recall Stirling's formula
\[n! \sim \sqrt{2\pi n}\left(\frac{n}{e}\right)^n.\]
We use $C_n$ to denote the $n$-th Catalan number:
\[C_n = \frac{1}{n+1}\binom{2n}{n}, \,\,\text{ and } \,C_n \sim
\frac{4^n}{\sqrt{\pi}n^{\frac{3}{2}}}\,.\]

\medskip

\subsection{Pattern avoidance}
Let $n$ and $m$ be positive integers with $m \leq n$, and let
\[\sigma = (\sigma(1), \sigma(2),\ldots,\sigma(n)) \in S_n,\]
and \emph{pattern} $\tau = (\tau(1),\tau(2),\ldots,\tau(m)) \in S_m$.
We say that $\sigma$ \emph{contains} $\tau$ if there exist indices
$i_1<i_2<\ldots<i_m$ such that $(\sigma(i_1),\sigma(i_2),\ldots,\sigma(i_m))$
is in the same relative order as $(\tau(1),\tau(2),\ldots,\tau(m))$.
If $\sigma$ does not contain $\tau$ then we say $\sigma$ is $\tau$-\emph{avoiding},
or \emph{avoiding pattern~$\tau$}. In this paper we use only $\tau \in S_3$; to simplify
the notation we use $\textbf{123}$ and $\textbf{132}$ to denote patterns $(1,2,3)$ or
$(1,3,2)$, respectively. For example, $\sigma = (2,4,5,1,3)$ contains~$\textbf{132}$,
since the subsequence $(\sigma(1),\sigma(2),\sigma(5)) = (2,4,3)$ has the same
relative order as $(1,3,2)$. However, $\sigma = (5,3,4,1,2)$ is \textbf{132}-avoiding.

Denote by $\CS_n(\pi)$ the set of $\pi$-avoiding permutations in~$S_n$.
For the case of patterns of length~3, it is known that regardless of the
pattern $\pi \in S_3$, we have $\vert \mathcal{S}_n(\pi) \vert = C_n\,$.
\begin{thm}[MacMahon, Knuth]\label{SamePattern}
For all $\pi \in S_3$, we have $\vert S_n(\pi) \vert = C_n$.
\end{thm}

While the equalities
$\vert \CS_n(132) \vert = \vert \CS_n(231) \vert= \vert \CS_n(213) \vert =
\vert \CS_n(312) \vert$ and $\vert \CS_n(123) \vert = \vert \CS_n(321) \vert$
are straightforward, the fact that
$\vert \CS_n(132) \vert = \vert \CS_n(123)\vert$ is more involved.

\medskip

\subsection{Symmetries}
Recall that $P_n(j,k)$ and $Q_n(j,k)$ denote the number of permutations in $\CS_n(123)$
and $\CS_n(132)$, respectively, of size $n$ that have $j$ in the $k$-th position.
In this section we discuss the symmetries of such permutations.

\begin{prop}\label{Symmetry} For all $n,j,k$ positive integers such that $1 \leq j,k \leq n$,
we have $P_n(j,k)=P_n(k,j)$ and $Q_n(j,k)=Q_n(k,j)$. Also, $P_n(j,k) = P_n(n+1-k,n+1-j)$,
for all $k,j$ as above.
\end{prop}

The proposition implies that we can interpret $P_n(j,k)$ as either
\[ \vert \{\sigma \in \CS_n(123)\,\text{s.t.}\,\sigma(j)=k\}\vert\,\,\text{ or }
\,\,\vert\{\sigma \in \CS_n(123)\,\text{s.t.}\,\sigma(k)=j\}\vert,\]
and we use both formulas throughout paper. The analogous statement holds
with $Q_n(j,k)$ as well.  Note, however, that $Q_n(j,k)$ is not necessarily
equal to $Q_n(n+1-k,n+1-j)$; for example, $Q_3(1,2) = 2$ but $Q_3(2,3) = 1$.
In other words, there is no natural analogue of the second part of the
proposition for~$Q_n(j,k)$, even asymptotically, as our results will
show in the next section.

\proof
The best way to see this is to consider permutation matrices.
Observe that $P_n(j,k)$ counts the number of permutation matrices
$A = (a_{rs})$ which have $a_{kj}=1$, but which have no row indices
$i_1<i_2<i_3$ nor column indices $j_1<j_2<j_3$ such that
$a_{i_1j_1}=a_{i_2j_2}=a_{i_3j_3}=1$. If $A$ is such a matrix,
then $B = A^T$ is also a matrix which satisfies the conditions
for $P_n(k,j)$, since $b_{jk}=1$ and $B$ has no indices which lead to
the pattern \textbf{123}. This transpose map is clearly a bijection,
so we have $P_n(j,k)=P_n(k,j)$.

Similarly, since any \textbf{132} pattern in a permutation matrix~$A$
will be preserved in $B = A^T$, we have $Q_n(j,k)=Q_n(k,j)$. Finally,
observe that $P_n(j,k) = P_n(n+1-k,n+1-j)$, since $\sigma$ is
\textbf{123}-avoiding if and only if
$\rho = (n+1-\sigma(n),n+1-\sigma(n-1),\ldots,n+1-\sigma(1))$
is \textbf{123}-avoiding. \qed


\medskip

\subsection{Maxima and minima} Here we find all maxima and minima of
matrices $P_n(\cdot,\cdot)$ and $Q_n(\cdot,\cdot)$.  We separate the
results into two propositions.

\begin{prop}\label{Pextreme} For all $n\ge 3$, the value of $P_n(j,k)$
is minimized
when $(j,k)=(1,1)$ or $(n,n)$. Similarly, $P_n(j,k)$ is maximized when
\[(j,k) = (1,n),(1,n-1),(2,n)\,\,\text{ or }(n-1,1),(n,1),(n,2).\]
\end{prop}
\proof
For any $n$, the only $\sigma \in \CS_n(123)$ with $\sigma(1)=1$ is
$\sigma = (1,n,n-1,\ldots,3,2)$. This implies that $P_n(1,1)=1$.
Similarly, $P_n(n,n) = 1$, since the only such permutation is
$(n-1,n-2,\ldots,2,1,n)$.

For every $j,k\leq n$, the maximum possible value of $P_n(j,k)$ is $C_{n-1}$,
since the numbers from 1 to $n$ excluding $j$ must be $\textbf{123}$-avoiding
themselves. Let us show that
\[P_n(1,n)=P_n(2,n)=P_n(1,n-1)=C_{n-1},\]
proving that this maximum is in fact achieved by the above values of $j$ and~$k$.

If $\sigma \in \CS_n(123)$ has $\sigma(1)=n$, then $n$ cannot be part of a
$\textbf{123}$-pattern, since it is the highest number but must be the smallest
number in the pattern. Therefore, any $\sigma \in \CS_{n-1}(123)$ can be
extended to a permutation $\tau \in \CS_n(123)$ in the following way:
let $\tau(1) = n$, and let $\tau(i)=\sigma(i-1)$ for $2 \leq i \leq n$.
Since $\vert\CS_{n-1}(123)\vert = C_{n-1}$, we have $P_n(1,n)=C_{n-1}$.
Similarly, if $\sigma(2)=n$, then $n$ cannot be part of a $\textbf{123}$-pattern,
so $P_n(2,n)=C_{n-1}$. The same is true if $\sigma(1)=(n-1)$, so $P_n(1,n-1)=C_{n-1}$.
By Proposition$~\ref{Symmetry}$, we also have
\[P_n(n-1,1)=P_n(n,1)=P_n(n,2)=C_{n-1},\]
as desired.
\qed

\medskip

\begin{prop}\label{Qextreme}  For all $n\ge 4$, the value of
$Q_n(j,k)$ is minimized when $(j,k)=(1,1)$. Similarly,
$Q_n(j,k)$ is maximized when
\[(j,k) = (1,n),(1,n-1),(n-1,1),(n,1),\,\,\text{ or }\,\,(n,n).\]
\end{prop}
\proof
For $(j,k)=(1,1)$, the only \textbf{132}-avoiding permutation is
$\sigma = (1,2,3,\ldots,n-1,n)$. Therefore, $Q_n(1,1)=1$ for all $n$.

For the second part, we use the same reasoning as in Proposition~\ref{Pextreme},
except for $(j,k)=(n,n)$. For $(j,k) = (n,n)$, we have $Q_n(n,n)=C_{n-1}$ as well,
since $n$ in the final position cannot be part of a \text{132}-pattern.
Observe that unlike $P_n(2,n)$, $Q_n(2,n) < C_{n-1}$, since $\sigma(2)=n$
requires $\sigma(1)=n-1$, in order to avoid a \textbf{132}-pattern.
\qed

\bigskip

\section{Main results}\label{s:main}
In this section we present the main results of the paper.
\subsection{Shape of \textbf{123}-avoiding permutations}
Let $0 \leq a,b \leq 1, 0 \leq \alpha < 1$, and $c \in \rr$ s.t.$~c \neq 0$ for
$\alpha \neq 0$ be fixed constants. Recall that $P_n(j,k)$ is the number of permutations
$\sigma \in \CS_n(123)$ with $\sigma(j)=k$. Define
\[F(a,b,c,\alpha) = \sup{\left\{d \in \mathbb{R}_+ \Big\vert \lim_{n \to \infty}
\frac{n^dP_n(an-cn^{\alpha},bn-cn^{\alpha})}{C_n} < \infty\right\}} \,\,\text{for }\alpha \neq 0 \,
\text{ or }a+b \neq 1,\]
and \[F(a,b,c,\alpha) =
\sup{\left\{d \in \mathbb{R}_+ \Big\vert \lim_{n \to \infty}
\frac{n^dP_n(an-cn^{\alpha}+1,bn-cn^{\alpha})}{C_n} < \infty\right\}} \,\, \text{for }\alpha = 0 \, \text{ and }a+b=1.
\]
Similarly, let
\[L(a,b,c,\alpha) =
\lim_{n \to \infty} \frac{n^{F(a,b,c,\alpha)}P_n(an-cn^{\alpha},bn-cn^{\alpha})}{C_n} \,\text{ for }\alpha \neq 0 \text{ or } a+b \neq 1,\]
and \[L(a,b,c,\alpha) =
\lim_{n \to \infty} \frac{n^{F(a,b,c,\alpha)}P_n(an-cn^{\alpha}+1,bn-cn^{\alpha})}{C_n} \,\text{ for }\alpha = 0 \text{ and }a+b=1,\]
defined for all $a,b,c,\alpha$ as above, for which $F(a,b,c,\alpha) < \infty$;
let $L$ be undefined otherwise.
\\
\begin{thm}\label{T1} For all $\,0 \leq a,b \leq 1, c \in \rr$ and $0 \leq \alpha < 1$,
we have \[F(a,b,c,\alpha) = \begin{cases}
\infty & \text{ if }\,a+b \neq 1\,,\\
\infty & \text{ if }\,a+b = 1,\,\, c \neq 0,\,\,\alpha > \frac{1}{2}\,,\\
\frac{3}{2} & \text{ if }\,a+b=1,\,\,c=0\,,\\
\frac{3}{2}-2\alpha & \text{ if }\,a+b=1,\,\, c \neq 0,\,\,\alpha \leq \frac{1}{2}\,.\\
\end{cases}\]
\end{thm}
Here $F(a,b,c,\alpha) = \infty$ means that
$P_n(an-cn^{\alpha},bn-cn^{\alpha}) = o(C_n/n^d)$,
for all $d>0$. The following result proves the exponential decay
 of these probabilities.

\begin{thm}\label{T5}
Let $0 \leq a,b \leq 1$ s.t. $a+b \neq 1,c \in \rr$, and  $0 < \alpha < 1$.
Then, for~$n$ large enough, we have
\[\frac{P_n(an-cn^{\alpha},bn-cn^{\alpha})}{C_n} < \ve^n,\]
where $\ve = \ve(a,b,c,\alpha)$ is independent of $n$, and $0 < \ve < 1$.
Similarly, let $0 \leq a \leq 1, c \neq 0$, and $ \frac{1}{2} < \alpha < 1$.
Then, for~$n$ large enough, we have
\[\frac{P_n(an-cn^{\alpha},(1-a)n-cn^{\alpha})}{C_n} < \ve^{n^{2\alpha-1}},\]
where $\ve=\ve(a,c,\alpha)$ is independent of $n$ and $0<\ve<1$.
\end{thm}

These theorems compare the growth of $P_n(an-cn^{\alpha},bn-cn^{\alpha})$
to the growth of$~C_n$. Clearly,
\[\sum_{j=1}^n\,\, P_n(j,k) = \vert \CS_n(123) \vert = C_n\,\,\,\text{
for all }\,1 \leq k \leq n\,\,.\]
Therefore, if $P_n(j,k)$ were constant across all values of $j,k$ between 1 and $n$,
we would have $P_n(j,k) = C_n/n$ for all $1 \leq j,k \leq n$.
Theorem~\ref{T1} states that for $0 \leq a,b \leq 1$, $a+b \neq 1$, we have
$P_n(an,bn) = o(C_n/n^d)$, for every $d \in \mathbb{R}$. For $a+b=1$,
we have $P_n(an,bn) = \Theta( C_n/n^{\frac{3}{2}})$.
Theorem~\ref{T1} is in fact stating slightly more. When we consider
$P_n(an-cn^{\alpha},bn-cn^{\alpha})$ instead of $P_n(an,bn)$,
we have
\[P_n(an-cn^{\alpha},bn-cn^{\alpha}) = \Theta\left(C_n/n^{\frac{3}{2}-2\alpha}\right),\]
for all $\alpha \leq \frac{1}{2}$. This relationship can be seen in
Figures~\ref{QPicture} and~\ref{QDetailed}.
\begin{figure}[ht!]
	\centering
\begin{tikzpicture}[scale=1,line/.style={&gt;=latex}]

  \begin{scope}
    \draw[scale = 3] (0, 0) grid (1, 1);
    \node[left] at (0,3) {0};
    \node[left] at (0,0) {1};
    \node[above] at (3,3) {1};
    \node[above] at (4,3) {a};
    \node[left] at (0,-1) {b};
    \node[left] at (1.1,2) {$\gamma$};
    \node[right] at (5,2) {$\gamma : \{a+b=1-\frac{c}{\sqrt{n}}\}$};
    \node[right] at (5,1) {$\gamma' : \{a+b=1+\frac{c}{\sqrt{n}}\}$};
    \node[right] at (1.85,1) {$\gamma'$};
    \shade[thick, left color=red, right color=white, shading angle = 45] (0,0) to [out=75,in=195] (3,3) to [out=224,in=46](0,0);
    \shade[thick, left color=red, right color=white, shading angle = 225] (0,0) to [out=44,in=226](3,3) to [out=255,in=15] (0,0);

    \draw[->] (0,3) -- (4,3);
    \draw[->] (0,3) to (0,-1);

    \end{scope}

\end{tikzpicture}
\caption{Region where $P_n(an,bn) \sim C_n/n^d$ for some $d$.}
\label{QPicture}
\end{figure}
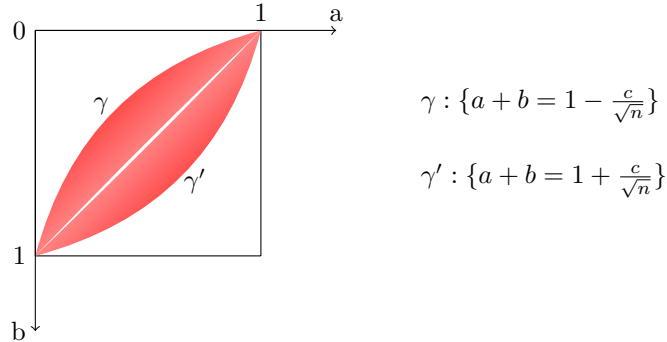
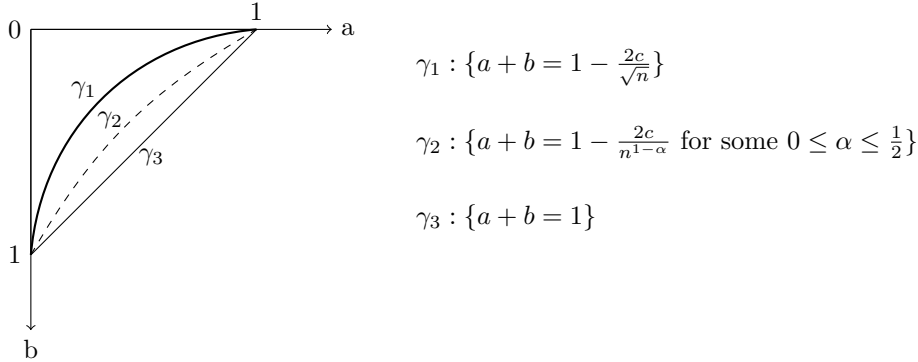
\begin{figure}[ht!]
	\centering
\begin{tikzpicture}[scale=1]

  	\begin{scope}
	\draw (0,0) to (0,3);
	\draw (0,3) to (3,3);
	\draw (0,0) to (3,3);
	\node[left] at (0,3) {0};
    	\node[left] at (0,0) {1};
    	\node[above] at (3,3) {1};
	\node[right] at (5,2.5) {$\gamma_1: \{a+b=1-\frac{2c}{\sqrt{n}}\}$};
	\node[left] at (1,2.2) {$\gamma_1$};
	\node[right] at (5,1.5) {$\gamma_2: \{a+b=1-\frac{2c}{n^{1-\alpha}}$    for some   $0 \leq \alpha \leq \frac{1}{2}\}$};
	\node[left] at (1.35,1.8) {$\gamma_2$};
	\node[left] at (1.9,1.3) {$\gamma_3$};
	\node[right] at (5,0.5) {$\gamma_3: \{a+b=1\}$};
	\draw[thick] (0,0) to [out=85,in=185] (3,3);
	\draw[dashed] (0,0) to [out=60,in=210] (3,3);
	\draw[->] (0,3) to (4,3);
    	\draw[->] (0,3) to (0,-1);
	\node[right] at (4,3) {a};
    	\node[below] at (0,-1) {b};
	\end{scope}

\end{tikzpicture}
\caption{Region where $P_n(an,bn) \sim C_n/n^d$ for some $d$.}
\label{QDetailed}
\end{figure}

In Figure$~\ref{QDetailed}$, on $\gamma_1$, we have
\[P_n\left(an-c\sqrt{n},(1-a)n-c\sqrt{n}\right) = \Theta\left(\frac{C_n}{\sqrt{n}}\right).\]
On $\gamma_2$, where
\[a+b=1-\frac{2c}{n^{1-\alpha}}\,,\,\,\text{ for some }\,\,0 \leq \alpha \leq \frac{1}{2},\]
we have
\[P_n(an-cn^{\alpha},(1-a)n-cn^{\alpha}) = \Theta\left(\frac{C_n}{n^{\frac{3}{2}-2\alpha}}\right).\]
On $\gamma_3$, we have $P_n(an,(1-a)n)  = \Theta(C_n/n^{\frac{3}{2}})$.
Behavior is symmetric about the line $a+b=1$.

The following result is a strengthening of Theorem$~\ref{T1}$ in a different direction.
For $a,b,c$, and $\alpha$ as above, s.t.~$F(a,b,c,\alpha) < \infty$, we calculate the value
of~$L(a,b,c,\alpha)$.

\begin{thm}\label{T2}
For all $0 \leq a \leq 1, c \in \rr$, and $0 \leq \alpha \leq 1/2$,
we have
\[L(a,1-a,c,\alpha) = \begin{cases}
\xi(a,c) & \text{ if }\,\,c=0\,\,\text{ or }\,\,\alpha=0\,,\\
\eta(a,c) & \text{ if }\,\,c \neq 0\,\,\text{ and }0 < \alpha < \frac{1}{2}\,,\\
\eta(a,c) \kappa(a,c) & \text{ if }\,\,c \neq 0\,\,\text{ and }\alpha = \frac{1}{2}\,,\\
\end{cases}\]
where
\[\xi(a,c) = \frac{(2c+1)^2}{4\sqrt{\pi}(a(1-a))^{\frac{3}{2}}},\,\,\eta(a,c)
= \frac{c^2}{\sqrt{\pi}(a(1-a))^{\frac{3}{2}}} \,\,\, \text{ and } \,\,\kappa(a,c)
= \exp{\left[{\frac{-c^2}{a(1-a)}}\right]}.\]
\end{thm}

\smallskip

Let us note that for $\alpha=0$ or $c=0$ as in theorem, we actually evaluate
$$P_n(an-cn^{\alpha},(1-a)n-cn^{\alpha}+1) \ \ \, \text{rather than} \ \ P_n(an-cn^{\alpha},(1-a)n-cn^{\alpha}).
$$
We do this in order to ensure that we truly measure the distance away from the anti-diagonal
where $j+k=n+1$. This change only affects the asymptotic behavior of $P_n(\cdot,\cdot)$
when $\alpha=0$ or $c=0$.

\bigskip

\subsection{Shape of \textbf{132}-avoiding permutations}
Recall that $Q_n(j,k)$ is the number of permutations $\sigma \in \CS_n(132)$
with $\sigma(j) = k$. Let $a,b,c$ and $\alpha$ be defined as
in Theorem~\ref{T1}. Define
\[G(a,b,c,\alpha) = \sup{\left\{d \in \mathbb{R}_+ \Big\vert \lim_{n \to \infty}
\frac{n^dQ_n(an-cn^{\alpha},bn-cn^{\alpha})}{C_n} < \infty \right\}}.\]
Let
\[M(a,b,c,\alpha) =
\lim_{n \to \infty} \frac{n^{G(a,b,c,\alpha)}Q_n(an-cn^{\alpha},bn-cn^{\alpha})}{C_n}\,,\]
defined for all $a,b,c,\alpha$ as above for which $G(a,b,c,\alpha) < \infty$;
$M$ is undefined otherwise.

\begin{thm}\label{T3} For $0\le a,b \le 1$, $c \in \zz$ and $\alpha \ge 0$, we have \. $G(a,b,c,\alpha) = $
\[= \, \begin{cases}
\infty & \hskip -0.2 cm \text{ if } \,\,0 \leq a+b < 1,\\
\frac{3}{2} & \hskip -0.2cm \text{ if } \,\,1 < a+b < 2,\\
\frac{3}{2}\ts\alpha & \hskip -0.18cm \text{ if } \,\,a=b = 1, \, 0 < \alpha < 1, \, c \neq 0,\\
0 & \hskip -0.2cm \text{ if } \,\,a=b=1, \,  \alpha = 0,\\
\frac{3}{4} & \hskip -0.2cm \text{ if } \,\,a+b=1, \, c=0,\\
\end{cases} \ \ \, \begin{cases}
\infty & \hskip -0.2cm \text{ if } \,\,a+b = 1, \,\, \frac{1}{2}< \alpha < 1, \,\,c>0,\\
\frac{3}{4} & \hskip -0.2cm \text{ if } \,\,a+b = 1,\,\, 0 \leq \alpha \leq \frac{3}{8},\,\,c \neq 0,\\
\frac{3}{4} & \hskip -0.2cm \text{ if } \,\,a+b=1, \, \frac{3}{8} \leq \alpha \leq \frac{1}{2},\,\,c < 0,\\
\frac{3}{2}\ts\alpha & \hskip -0.18cm \text{ if } \,\,a+b=1, \, \frac{1}{2} < \alpha < 1,\,\,c<0,\\
\frac{3}{2}-2\alpha & \hskip -0.2cm \text{ if } \,\,a+b=1, \, \frac{3}{8} \leq \alpha \leq \frac{1}{2},\,\,c > 0.\\
\end{cases}\]
\end{thm}

As in Theorem~\ref{T1}, here $G(a,b,c,\alpha)=\infty$ means that
$Q_n(an-cn^{\alpha},bn-cn^{\alpha}) = o(C_n/n^d)$, for all $d>0$.
The following result proves exponential decay of these probabilities
(cf.~Theorem~\ref{T5}.)
\begin{thm}\label{T6}
Let $0 \leq a,b < 1$ such that $a+b<1,c \neq 0$, and $0 < \alpha < 1$.
Then, for~$n$ large enough, we have
\[\frac{Q_n(an-cn^{\alpha},bn-cn^{\alpha})}{C_n} < \ve^n,\]
where $\ve = \ve(a,b,c,\alpha)$ is independent of $n$, and $0<\ve<1$.
Similarly, let $0 \leq a \leq 1, 0 < c$, and $\frac{1}{2} < \alpha < 1$.
Then, for~$n$ large enough, we have
\[\frac{Q_n(an-cn^{\alpha},(1-a)n-cn^{\alpha})}{C_n} < \ve^{n^{2\alpha-1}},\]
where $\ve = \ve(a,c,\alpha)$ is independent of $n$, and $0<\ve<1$.
\end{thm}

The above theorems compare the relative growth rates of $Q_n(i,j)$ and$~C_n$, as $n \to \infty$.
Theorem~\ref{T3} states that for $a+b<1, Q_n(an,bn) = o(C_n/n^d)$ for all $d>0$.
For $1 < a+b < 2$, we have
\[Q_n(an,bn) = \Theta\left(\frac{C_n}{n^{\frac{3}{2}}}\right).\]
$Q_n(an,bn)$ is the largest when $a=b=1$ or when $a+b=1$.
The true behavior of $Q_n(i,j)$ described in Theorems$~\ref{T3} \text{ and }\ref{T4}$
takes the second order terms of $i$ and $j$ into account. In fact we have that
\[Q_n(n-cn^{\alpha},n-cn^{\alpha}) = \Theta\left(\frac{C_n}{n^{\frac{3}{2}\alpha}}\right)
\quad \text{ when }\alpha \leq \frac{1}{2}.\]
For $a+b=1$, the asymptotic behavior of $Q_n(an-cn^{\alpha},bn-cn^{\alpha})$ varies
through several regimes as $\alpha$ varies between 0 and~1, and $c$ varies between positive
and negative numbers. This relationship is illustrated in Figures~\ref{PPicture} and~\ref{PDetailed}.

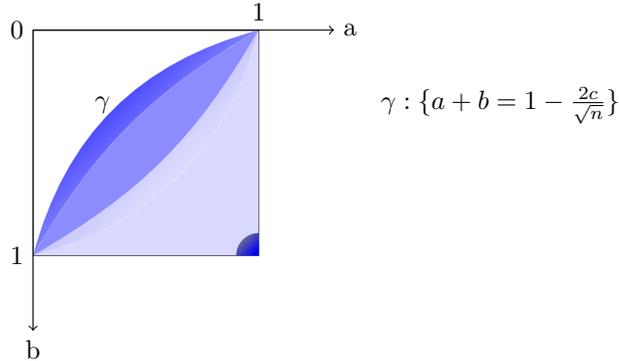
\begin{figure}[ht!]
	\centering
\begin{tikzpicture}[scale=1]

  \begin{scope}
    \draw[scale = 3] (0, 0) grid (1, 1);
    \node[left] at (0,3) {0};
    \node[left] at (0,0) {1};
    \node[above] at (3,3) {1};
    \node[right] at (4,3) {a};
    \node[below] at (0,-1) {b};
    \node[left] at (1.15,2) {$\gamma$};
    \node[right] at (4.5,2) {$\gamma : \{a+b=1-\frac{2c}{\sqrt{n}}\}$};
    \shade[thick, left color=blue, shading angle = 45] (0,0) to [out=75,in=195] (3,3) to [out=210,in=60] (0,0);
    \shade[thick, left color=blue!45, right color=blue!45] (0,0) to [out=60,in=210] (3,3) to [out=240,in=30] (0,0);
    \shade[thick, left color=blue!45, shading angle = 45] (0,0) to [out=30,in=240] (3,3) to [out=255,in=15] (0,0);
     \shade[thick, left color=blue!15, right color = blue!15] (0,0) to [out=15,in=255] (3,3) -- (3,0) -- (0,0);
     \shade[right color = blue, shading angle = 45] (2.7,0) to [out=90,in=180] (3,0.3) -- (3,0) -- (2.7,0);
    \draw[->] (0,3) to (4,3);
    \draw[->] (0,3) to (0,-1);

    \end{scope}

\end{tikzpicture}
\caption{Region where $Q_n(an,bn) \sim C_n/n^d$ for some $d$.}
\label{PPicture}
\end{figure}

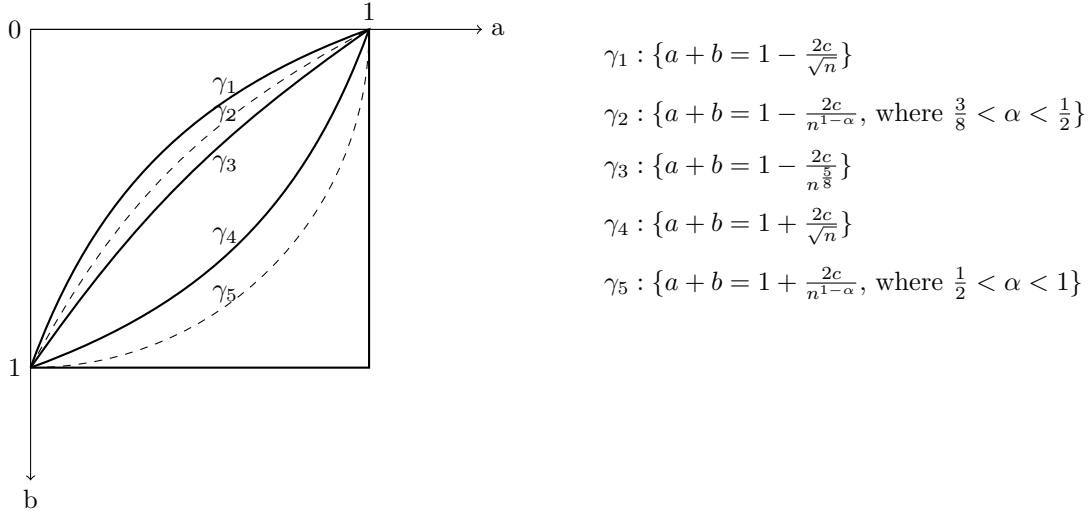
\begin{figure}[ht!]
	\centering
\begin{tikzpicture}[scale=1.5]

  	\begin{scope}
	\draw (0,0) to (0,3);
	\draw (0,3) to (3,3);
	\node[left] at (0,3) {0};
    	\node[left] at (0,0) {1};
    	\node[above] at (3,3) {1};
	\node[right] at (5,2.75) {$\gamma_1 : \{a+b=1-\frac{2c}{\sqrt{n}}\}$};
	\node[left] at (1.92,2.5) {$\gamma_1$};
	\node[left] at (1.92,2.25) {$\gamma_2$};
	\node[right] at (5,2.25) {$\gamma_2 : \{a+b=1-\frac{2c}{n^{1-\alpha}},$
                    where  $\frac{3}{8} < \alpha < \frac{1}{2}\}$};
	\node[left] at (1.92,1.81) {$\gamma_3$};
	\node[right] at (5,1.75) {$\gamma_3 : \{a+b=1-\frac{2c}{n^{\frac{5}{8}}}\}$};
	\node[left] at (1.92,1.17) {$\gamma_4$};
	\node[right] at (5,1.25) {$\gamma_4 : \{a+b=1+\frac{2c}{\sqrt{n}}\}$};
	\node[left] at (1.92,.65) {$\gamma_5$};
	\node[right] at (5,.75) {$\gamma_5 : \{a+b=1+\frac{2c}{n^{1-\alpha}},$
                    where  $\frac{1}{2} < \alpha < 1\}$};
	
	\draw[thick] (0,0) to [out=70,in=200] (3,3) to [out=250,in=20] (0,0);
	\draw[thick] (0,0) to [out=55,in=215] (3,3);
	\draw[dashed] (0,0) to [out=62,in=208] (3,3);
	\draw[dashed] (3,3) to [out=270,in=0] (0,0);
	\draw[->] (0,3) to (4,3);
    	\draw[->] (0,3) to (0,-1);
	\node[right] at (4,3) {a};
    	\node[below] at (0,-1) {b};
	\draw[thick] (0,0) to (3,0) to (3,3);
	\end{scope}

\end{tikzpicture}
\caption{Region where $Q_n(an,bn) \sim C_n/n^d$ for some $d$.}
\label{PDetailed}
\end{figure}

\noindent
In Figure~\ref{PDetailed}, on the curve~$\gamma_1$, we have
\[Q_n(an-c\sqrt{n},(1-a)n-c\sqrt{n}) = \Theta\left(\frac{C_n}{n^{\frac{1}{2}}}\right).\]
Similarly, on $\gamma_2$, we have
\[Q_n(an-cn^{\alpha},(1-a)n-cn^{\alpha}) = \Theta\left(\frac{C_n}{n^{\frac{3}{2}-2\alpha}}\right).\]
In the space between $\gamma_3$ and $\gamma_4$, we have
\[Q_n(an+k,(1-a)n+k) = \Theta\left(\frac{C_n}{n^{\frac{3}{4}}}\right),\]
where $-cn^{\frac{3}{4}} \leq k \leq cn^{\frac{1}{2}}$. Finally, on~$\gamma_5$, we have
\[Q_n(an+cn^{\alpha},(1-a)n+cn^{\alpha}) = \Theta\left(\frac{C_n}{n^{\frac{3}{2}\alpha}}\right).\]

As in Theorem$~\ref{T2}$, the following result strengthens Theorem~\ref{T3}
in a different direction. For $a,b,c,\alpha$ where $G(a,b,c,\alpha) < \infty$,
we calculate the value of $M(a,b,c,\alpha)$.
\begin{thm}\label{T4}
For $a,b,c,\alpha$ as above, we have \. $M(a,b,c,\alpha) = $
\[= \begin{cases}
v(a,b) & \hskip -0.3cm \text{ if } \,\,1 < a+b < 2,\\
w(c) & \hskip -0.3cm \text{ if } \,\,a=b=1, \, 0 < \alpha < 1, \\
u(c) & \hskip -0.3cm \text{ if } \,\,a=b=1, \, c\ge 0, \, \alpha = 0, \\
w(c) & \hskip -0.3cm \text{ if } \,\,a+b=1, \, c<0, \, \frac{1}{2} < \alpha,\\
x(a,c) & \hskip -0.3cm \text{ if } \,\,a+b=1, \, c<0, \, \alpha = \frac{1}{2},\\
\end{cases}
\quad
\begin{cases}
z(a) & \hskip -0.3cm \text{ if } \,\,a+b=1, \, 0 \leq \alpha < \frac{3}{8},\\
z(a) & \hskip -0.3cm \text{ if } \,\,a+b=1, \, c<0, \, \frac{3}{8} \leq \alpha < \frac{1}{2},\\
z(a) + y(a,c) & \hskip -0.3cm \text{ if } \,\,a+b=1, \, c>0,  \, \alpha = \frac{3}{8},\\
y(a,c) & \hskip -0.3cm \text{ if } \,\,a+b=1,\,  c>0, \, \frac{3}{8} < \alpha < \frac{1}{2},\\
y(a,c)\,\kappa(a,c) & \hskip -0.3cm \text{ if } \,\,a+b=1, \, c>0, \, \alpha = \frac{1}{2},\\
\end{cases}\]
where
$$
u(c) = \sum_{s=0}^c\left(\frac{s+1}{2c+1-s}\right)^2\binom{2c+1-s}{c+1}^2 4^{s-2c-1}, \quad
v(a,b) = \frac{1}{2\sqrt{\pi}(2-a-b)^{\frac{3}{2}}(a+b-1)^{\frac{3}{2}}}\., $$ 
$$w(c)  = \frac{1}{2^{\frac{5}{2}}\,c^{\frac{3}{2}}\sqrt{\pi}}\., \qquad
x(a,c) = \frac{1}{4\pi a^{\frac{3}{2}}(1-a)^{\frac{3}{2}}} \,
\int_0^{\infty}\frac{s^2}{(s+2c)^{\frac{3}{2}}}\. \exp{\left[\frac{-s^2}{4a(1-a)}\right]}ds,
$$
$$
y(a,c) =
\frac{2c^2}{\sqrt{\pi}a^{\frac{3}{2}}(1-a)^{\frac{3}{2}}}\., \qquad
z(a) = \frac{\Gamma(\frac{3}{4})}{2^{\frac{3}{2}}\pi a^{\frac{3}{4}}(1-a)^{\frac{3}{4}}}\.,
$$
and $\kappa(a,c)$ is defined as in Theorem~\ref{T2}.
\end{thm}

Observe that for $c=0$ or $\alpha=0$, values $Q_n(an-cn^{\alpha},(1-a)n-cn^{\alpha})$ behave
the same asymptotically as $Q_n(an-cn^{\alpha},(1-a)n-cn^{\alpha}+1)$. We explain this in more
detail in Lemma~\ref{SmallAlphaQ}. This is in contrast with the behavior of $P_n(\cdot,\cdot)$,
where we need to adjust when on the anti-diagonal. Note also that for $a=b=1$, $\al=0$ and $c=0$,
we have
$$
u(0) \. = \. \frac{1}{4} \. = \. \lim_{n\to \infty} \frac{Q_n(n,n)}{C_n} \. =
\. \lim_{n\to \infty} \frac{C_{n-1}}{C_n}\.,
$$
which holds since $Q_n(n,n) = C_{n-1}$  given in the proof of~Proposition~\ref{Qextreme}.
We prove the theorem in Section~\ref{s:132}.

\bigskip

\section{Analysis of \textbf{123}-avoiding permutations} \label{s:123}

\subsection{Combinatorics of Dyck Paths}
We say a Dyck path of length $2n$ is a path from $(0,0)$ to $(2n,0)$ in $\mathbb{Z}^2$
consisting of \emph{upsteps} (1,1) and \emph{downsteps} (1,-1) such that the path never
goes below the $x$-axis. We denote by $\cd_n$ the set of Dyck paths of length $2n$.
We can express a Dyck path $\gamma \in \cd_n$ as a word of length $2n$, where $u$
represents an upstep and $d$ represents a downstep.

Recall that $P_n(j,k)$ is the number of permutations $\sigma \in \CS_n(123)$ with
$\sigma(j) = k$. Let $f(n,k) = P_n(1,k)$ (or $P_n(k,1)$). Let $b(n,k)$ be the number
of lattice paths consisting of upsteps and downsteps from $(0,0)$ to $(n+k-2,n-k)$
which stay above the $x$-axis. Here $b(n,k)$ are the \emph{ballot numbers},
given by
\[b(n,k)=\frac{n-k+1}{n+k-1}\binom{n+k-1}{n}.\]

\begin{lem}\label{Pbij}
For all $1 \leq k \leq n$, we have $f(n,k) = b(n,k)$.
\end{lem}

\proof
We have that $f(n,k)$ counts the number of permutations $\sigma \in \CS_n(123)$
such that $\sigma(1) = k$. By the RSK-correspondence (see e.g.~\cite{B1,Sta}),
we have $f(n,k)$ counts the number of Dyck paths $\gamma \in \cd_n$ whose final
upstep ends at the point $(n+k-1,n+1-k)$. Remove the last upstep from path~$\gamma$,
and all the steps after it. We get a path $\gamma'$ from $(0,0)$ to $(n+k-2,n-k)$
which remains above the $x$-axis. These paths are counted by $b(n,k)$, and the map
$\gamma \to \gamma'$ is clearly invertible, so $f(n,k)=b(n,k)$, as desired.
\qed

\medskip

\begin{lem}\label{explicitP} For all $1 \leq j,k \leq n$, we have
\[P_n(j,k) = b(n-k+1,j)\,b(n-j+1,k), \,\,\,\text{ where }\,\,j+k \leq n+1.\]
Similarly, we have
\[P_n(j,k) = b(j,n-k+1)\,b(k,n-j+1),\,\,\,\text{ where }\,\,j+k > n+1.\]
\end{lem}

\proof
Let us show that the second case follows from the first case. Suppose $j+k>n+1$. By assuming the first case of the lemma, we have
$$\aligned
P_n(j,k) \, & = \, P_n(n+1-j,n+1-k) \\
& = \, b(n-(n+1-k)+1,n+1-j)\, b(n-(n+1-j)+1,n+1-k) \\
& = \, b(k,n-j+1)\,b(j,n-k+1),
\endaligned
$$
by Proposition~\ref{Symmetry}.
Therefore, it suffices to prove the lemma for $j+k \leq n+1$.

Let $j+k \leq n+1$, and let $\sigma$ be a \textbf{123}-avoiding permutation with $\sigma(j)=k$.
We use decomposition \. $\sigma \ts = \ts  \tau \ts k \ts \rho$, where
$$\tau = \{\sigma(1),\ldots,\sigma(j-1)\} \ \ \text{and} \ \, \rho = \{\sigma(j+1),\ldots,\sigma(n)\}\ts.
$$
We now show that $\sigma(i) > k$, for all $1 \leq i < j$.
Suppose $\sigma(i) < k$ for some $i<j$. Then there are at most $(j-2)$ numbers $x<j$ with $\sigma(x) > k$.
Since $\sigma(j)=k$, in total there are $(n-k)$ numbers $y$ such that $\sigma(y) > k$. Since $j-2 < n-k$,
there must be at least one number $z>j$ with $\sigma(z) > k$. However, this gives a
\textbf{123}~pattern consisting of $(i, j, z)$, a contradiction.
 Therefore, $\sigma(i) > k$, for all $1 \leq i < j$.

Consider the values of $\sigma$ within $\tau$. From above, the values within $\tau$ are
all greater than $k$. Given a possible $\tau$, the values within $\rho$ which are greater than~$k$
must be in decreasing order, to avoid forming a \textbf{123}-pattern starting with $k$.
Therefore, to count possible choices for $\tau$, it suffices to count possible orderings
within $\sigma$ of the numbers $x$ with $k \leq x \leq n$. The number of such orderings
is $b(n-k+1,j)$, since the smallest number is in the $j$-th position. Therefore,
there are $b(n-k+1,j)$ possible choices for$~\tau$.

Now consider the values of $\sigma$ within $\{k\} \cup \rho$. We have $(n-j+1)$ numbers to order,
and $(k-1)$ of them are less than $k$. Our only restriction on $\rho$ is that we have no
\textbf{123}-patterns. There are $b(n-j+1,k)$ of these orderings, since the $k$-th smallest number
is in the $1$-st position. Therefore, we have $b(n-j+1,k)$ possible choices for$~\rho$.

Once we have chosen $\tau$ and $\rho$, this completely determines the permutation~$\sigma$.
Therefore, there are $b(n-j+1,k)\ts b(n-k+1,j)$ choices of such $\sigma$, as desired. \qed

\begin{Example}\label{exQ}
\rm Let us compute $P_7(4,3)$, the number of permutations $\sigma \in \CS_7(123)$ with
$\sigma(4)=~3.$ The numbers 1 and 2 must come after 3 in any such permutation,
since otherwise a \textbf{123} will be created with 3 in the middle. There are
\[b(7-3+1,4) = \frac{2}{8}\binom{8}{5} =14\]
ways to order the numbers between 3 and 7, shown here:
\[\begin{array}{ccccccc}
(47635)&(54736)&(57436)&(57634)&(64735)&(65437)&(65734)\\
(67435)&(67534)&(74635)&(75436)&(75634)&(76435)&(76534).\\
\end{array}\]
For each of these, there are \[b(7-4+1,3)= \frac{2}{6}\binom{6}{4} = 5\] ways to place
the numbers between 1 and~3, shown here:
\[\begin{array}{ccccc}
(***\,3\,1*2)&(***\,3\,2\,1*)&(***\,3\,2*1)&(***\,3*1\,2)&(***\,3*2\,1),\\
\end{array}\]
where the asterisks represent the positions of $4,5,6$, and $7$.
In total, we have $P_7(4,3)=b(5,4)b(4,3)=(14)(5)=70$.
\end{Example}

\medskip
\subsection{Proof of theorems~\ref{T1},~\ref{T5}, and~\ref{T2}}
The proof follows from several lemmas: one technical lemma and one lemma for each case from Theorem~\ref{T1}.

\smallskip
Let $h:[0,1]^2 \to \mathbb{R}$ be defined as
\[h(a,b) = \frac{(1-a+b)^{(1-a+b)}(1-b+a)^{(1-b+a)}}{a^a(1-a)^{(1-a)}b^b(1-b)^{(1-b)}}.\]
\begin{lem}[Technical lemma]\label{Qbase} We have
\[h(a,b) \leq 4, \,\,\, \text{ for all }\,\, 0 \leq a,b \leq 1.\]
Moreover, $h(a,b)=4$ if and only if $\,b=1-a$.
\end{lem}
\proof
Observe that $h(a,1-a) = 4$. Furthermore, we consider the partial derivatives of~$h$
with respect to $a$ and $b$. We find that $h$ has local maxima at each point where $b=1-a$.
In fact these are the only critical points within $[0,1]^2$. We omit the details.\footnote{The
proof follows similar (and even somewhat simplified) steps as the proof of Lemma~\ref{Pbase}.}
\qed
\medskip

\begin{lem}[First case]\label{Neq1}
Let $a,b \in [0,1], c \neq 0$, and $0 \leq \alpha < 1$, such that $a+b \neq 1$.
Then $F(a,b,c,\alpha) = \infty$. Moreover, for $n$ sufficiently large, we have
$P_n(an-cn^{\alpha},bn-cn^{\alpha})/C_n < \ve^n,$ where $\ve$ is independent of~$n$ and $0<\ve<1$.
\end{lem}
\proof
By lemmas~\ref{Pbij} and~\ref{explicitP}, we have
\[P_n(an-cn^{\alpha},bn-cn^{\alpha}) = b(n-(bn-cn^{\alpha})+1,an-cn^{\alpha})\,\,
b(n-(an-cn^{\alpha})+1,bn-cn^{\alpha}) \]
\[= \frac{\left(n(1-a-b)+2cn^{\alpha}+2\right)^2}{n^2(1-b+a)(1-a+b)}
\binom{n(1-b+a)}{n-(bn-cn^{\alpha})+1}\binom{n(1-a+b)}{n-(an-cn^{\alpha})+1}\,.\]
Applying Stirling's formula gives
\[P_n(an-cn^{\alpha},bn-cn^{\alpha}) \sim r(n,a,b)\cdot h(a,b)^n,\]
where \. $r(n,a,b) \ts =$
\[= \.
\frac{(n(1-a-b)+2cn^{\alpha}+2)^2(an-cn^{\alpha})(bn-cn^{\alpha})
\sqrt{(1-a+b)(1-b+a)}}{2\pi n^3(n(1-a)+cn^{\alpha}+1)(n(1-b)+cn^{\alpha}+1)(1-a+b)(1-b+a) \sqrt{ab(1-a)(1-b)}}\ts.
\]
Using
\[C_n \sim \frac{4^n}{\sqrt{\pi}n^{\frac{3}{2}}}\,,\] we obtain \[\frac{n^dP_n(an-cn^{\alpha},bn-cn^{\alpha})}{C_n}
\sim \sqrt{\pi}\,\,n^{d+\frac{3}{2}}\,r(n,a,b)\,h(a,b)^n\,4^{-n}.\]
Clearly, for $h(a,b)<4$, the r.h.s.~$\to 0$ as $n \to \infty$, for all $d \in \mathbb{R}_+$.
By Lemma$~\ref{Qbase}$, we have $h(a,b)<4$ unless $b=1-a$. Therefore, since $a+b \neq 1$,
we have $F(a,b,c,\alpha) = \infty$. Also, when $n$ is large enough,
\[\frac{P_n(an-cn^{\alpha},bn-cn^{\alpha})}{C_n} < \left(\frac{h(a,b)+4}{8}\right)^n,\]
as desired. \qed
\medskip

\begin{lem}[Second case]\label{BigAlpha}
For all $a \in (0,1), 0 < c,$ and $\frac{1}{2} < \alpha < 1$, we have $F(a,1-a,c,\alpha) =~\infty$.
Moreover, for $n$ large enough, we have
\[\frac{P_n(an-cn^{\alpha},(1-a)n-cn^{\alpha})}{C_n} < \ve^{n^{2\alpha-1}},\]
where $\ve$ is independent of $n$ and $0<\ve<1$.
\end{lem}
\proof
Let $k=cn^{\alpha}$. Evaluating $P_n(an-k,(1-a)n-k)$, we have
\[P_n(an-k,(1-a)n-k) = \frac{(2k+2)^2}{(2(1-a)n)(2an)}\binom{2(1-a)n}{(1-a)n+k+1}\binom{2an}{an+k+1}.\]
Using Stirling's formula and simplifying this expression gives
\[P_n(an-k,(1-a)n-k) \sim \frac{(k+1)^2}{\pi(a(1-a))^{\frac{3}{2}}n^3}\,\,4^n
\left(\frac{an}{an+k}\right)^{an+k}\left(\frac{an}{an-k}\right)^{an-k} \]\[\hspace{1.5in}
\qquad\qquad\quad\times\left(\frac{(1-a)n}{(1-a)n+k}\right)^{(1-a)n+k}\left(\frac{(1-a)n}{(1-a)n-k}\right)^{(1-a)n-k}.\]
Clearly,
\[ \ln{\left[\left(\frac{an}{an+k}\right)^{an+k}\left(\frac{an}{an-k}\right)^{an-k}\right]}
\sim \frac{-k^2}{an}\,\,\,\,\text{ as }\,\,n \to \infty.\]
Therefore,
\[\frac{n^dP_n(an-k,(1-a)n-k)}{C_n} \sim \frac{n^d(k+1)^2}{\sqrt{\pi}(a(1-a)n)^{\frac{3}{2}}}\,\,
\exp{\left[{\frac{-k^2}{a(1-a)n}}\right]}.\]
Substituting $k \gets cn^{\alpha}$, gives
\[\frac{n^dP_n(an-k,(1-a)n-k)}{C_n} \sim \frac{c^2n^d}{\sqrt{\pi}(a(1-a))^{\frac{3}{2}}n^{\frac{3}{2}-2\alpha}}\,\,
\exp{\left[{\frac{-c^2}{a(1-a)n^{1-2\alpha}}}\right]}.\]
For $\alpha > \frac{1}{2}$, this expression $\to$ 0 as $n \to \infty$, for all~$d$.
This implies that $F(a,1-a,c,\alpha)=\infty$.
In fact, we have also proved the second case of Theorem~\ref{T4}, as desired.
\qed
\medskip

\begin{lem}[Third case]\label{OffDiagonal}
For all $a \in (0,1), c>0$, and $\alpha \in [0,1]$, we have
\[F(a,1-a,0,\alpha)=F(a,1-a,c,0)=\frac{3}{2}\,.\]
Furthermore, we have $L(a,1-a,c,0) = \xi(a,c)$.
\end{lem}
\proof
In this case, to ensure that $cn^{\alpha}$ measures the distance from the anti-diagonal,
we need to analyze $P_n(an-cn^{\alpha},(1-a)n-cn^{\alpha}+1)$.
Evaluating as in Lemma~\ref{BigAlpha}, gives
\[\frac{n^dP_n(an-cn^{\alpha},(1-a)n-cn^{\alpha}+1)}{C_n} \. \sim \.
\frac{n^{d-\frac{3}{2}}(2k+1)^2}{4\sqrt{\pi}(a(1-a))^{\frac{3}{2}}} \.
\exp{\left[\frac{-k^2}{a(1-a)n}\right]}.\]
For $c=0$, we get $F(a,1-a,0,\alpha) = 3/2$ and $L(a,1-a,0,\alpha) = \xi(a,0)$.
For $\alpha=0$, we get $F(a,1-a,c,0) = 3/2$ and $L(a,1-a,c,0)=\xi(a,c)$,
as desired.
\qed
\medskip

\begin{lem}[Fourth case]\label{SmallAlpha}
For all $a \in (0,1), c > 0$ and $0 < \alpha \leq \frac{1}{2}$, we have
$$F(a,1-a,c,\alpha) \. = \. \frac{3}{2}-2\alpha\ts.
$$
Furthermore, for $0 < \alpha < \frac{1}{2}$,
we have
$$L(a,1-a,c,\alpha) \. = \. \eta(a,c) \quad \text{and} \ \ \.
L\left(a,1-a,c,\frac{1}{2}\right) \. = \. \eta(a,c)\ts \kappa(a,c)\ts,
$$
where $\eta(a,c)$ and $\kappa(a,c)$ are defined as in Theorem~\ref{T2}.
\end{lem}
\proof
As in Lemma~\ref{BigAlpha}, we have
\[\frac{n^dP_n(an-k,(1-a)n-k)}{C_n} \. \sim \.
\frac{c^2n^d}{\sqrt{\pi}(a(1-a))^{\frac{3}{2}}n^{\frac{3}{2}-2\alpha}}\,\,
\exp{\left[{\frac{-c^2}{a(1-a)n^{1-2\alpha}}}\right]}.
\]
We can rewrite this expression as
\[ \frac{c^2n^d}{\sqrt{\pi}(a(1-a))^{\frac{3}{2}}n^{\frac{3}{2}-2\alpha}}\,\,
\exp{\left[{\frac{-c^2}{a(1-a)n^{1-2\alpha}}}\right]} \sim \eta(a,c)\,\, n^{d-(\frac{3}{2}-2\alpha)}
\exp{\left[{\frac{-c^2}{a(1-a)n^{1-2\alpha}}}\right]}.\]
For $\alpha < \frac{1}{2}$, we clearly have
\[\exp{\left[{\frac{-c^2}{a(1-a)n^{1-2\alpha}}}\right]} \to 1 \,\,\text{ as }\,\,n \to \infty,\]
so $F(a,1-a,c,\alpha) = \frac{3}{2}-2\alpha$ and $L(a,1-a,c,\alpha)=\eta(a,c)$.
For $\alpha = \frac{1}{2}$, by the definition of $\kappa(a,c)$, we have
\[\exp{\left[{\frac{-c^2}{a(1-a)n^{1-2\alpha}}}\right]} \to \kappa(a,c),\,\,\text{ as }\,\,n \to \infty,\]
so $F(a,1-a,c,1/2)=1/2$, and $L(a,1-a,c,1/2)=\eta(a,c)\kappa(a,c)$, as desired.
\qed
\\
\medskip

Let us emphasize that the results of the previous two lemmas hold for $c<0$ as well as $c>0$.
This is true by the symmetry of $P_n(j,k)$ displayed in Lemma~\ref{Symmetry}, and since~$c$
only appears in the formulas for $\eta(a,c)$ and $\kappa(a,c)$ as $c^2$. Therefore,
we have proven all cases of Theorems~\ref{T1} and~\ref{T2}.

\bigskip

\section{Analysis of \textbf{132}-avoiding permutations}\label{s:132}

\subsection{Combinatorics of Dyck paths}\label{132Dyck}
We recall a bijection $\varphi$ between $\CS_n(132)$ and~$\cd_n$, which we then use
to derive the exact formulas for $Q_n(j,k)$.  This bijection is equivalent to that
in~\cite{EP}, itself a variation on a bijection in~\cite{Kra} (see also~\cite{B1,Kit}
for other bijections between these combinatorial classes).

Given $\gamma \in \cd_n,$ for each downstep starting at point $(x,y)$ record~$y$,
the \emph{level} of $(x,y)$. This defines $y_\gamma = (y_1,y_2,\ldots,y_n)$.

We create the \textbf{132}-avoiding permutation by starting with a string
$\{n,n-1,\ldots,2,1\}$ and removing elements from the string one at a time
each from the $y_i$-th spot in the string, creating a permutation~$\varphi(\gamma)$.
Suppose this permutation contains a \textbf{132}-pattern, consisting of elements $a,b$, and~$c$
with $a<b<c$. After $a$ has been removed from the string, the level in the string must
be beyond $b$ and~$c$. Since we can only decrease levels one at a time, we must remove
$b$ before removing~$c$, a contradiction. Therefore, the map~$\varphi$ is well-defined,
and clearly one-to-one. By Theorem~\ref{SamePattern}, this proves that~$\varphi$ is the
desired bijection.

\begin{Example}\label{exDyck}
\rm Take the Dyck path $\gamma = (uuduuddudd)$. Then $z_{\gamma} = (2,3,2,2,1)$,
as seen in Figure $\ref{DyckPathExample}$. We then create our \textbf{132}-avoiding
permutation $\varphi(\gamma)$ by taking the string $\{5,4,3,2,1\}$ and removing
elements one at a time. First we remove the $2$-nd element $(4)$, then we remove
the $3$-rd element from the remaining list $\{5,3,2,1\}$, which is$~2$, then the
$2$-nd from the remaining list $\{5,3,1\}$, which is$~3$, then the $2$-nd from
$\{5,1\}$, which is$~1$, then the last element $(5)$, and we obtain
$\varphi(\gamma) = (4,2,3,1,5)$.

\begin{figure}[ht!]
	\centering

\begin{tikzpicture}[scale=0.5]

  \begin{scope}
    \draw (0, 0) grid (10, 3);
    \draw[thick] (0,0) -- (2,2);
    \draw[ultra thick] (2,2) -- (3,1);
    \draw[thick] (3,1) -- (5,3);
    \draw[ultra thick] (5,3) -- (7,1);
    \draw[thick] (7,1) -- (8,2);
    \draw[ultra thick] (8,2) -- (10,0);
    \node[left] at (0,1) {Level 1};
    \node[left] at (0,2) {Level 2};
    \node[left] at (0,3) {Level 3};
    \end{scope}

\end{tikzpicture}
\caption{Dyck Path $\gamma$ with downsteps at $y_\gamma = (2,3,2,2,1)$.}
\label{DyckPathExample}
\end{figure}
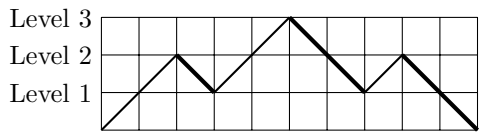
\end{Example}

Let $g(n,k)$ be the number of permutations $\sigma \in \CS_n(132)$ with
$\sigma(1) = k$, so $g(n,k) = Q_n(1,k)$. Recall that since $Q_n(j,k) = Q_n(k,j)$,
we can also think of $g(n,k)$ as the number of permutations $\sigma \in \CS_n(132)$
with $\sigma(k)=1$. Let $b(n,k)$ denote the ballot numbers as in Lemma~\ref{Pbij}.
\begin{lem}\label{ballot}
For all $1 \leq k \leq n$, we have $g(n,k) = b(n,k).$
\end{lem}

\proof
Let $\sigma \in \CS_n(132)$ with $\sigma(1) = k$. Using bijection~$\varphi$,
we find that $\varphi^{-1}(\sigma)$ is a Dyck path with its final upstep from
$(n+k-2,n-k)$ to $(n+k-1,n+1-k)$. The result now follows from the same logic
as in the proof of Lemma~\ref{Pbij}. \qed

\medskip

\begin{lem}\label{explicitQ} For all $1 \leq j,k \leq n$,
\[Q_n(j,k)\, = \,\sum_r \,b(n-j+1,k-r)\,\,b(n-k+1,j-r)\,C_r,\]
where the summation is over values of $r$ such that
\[\max{\{0,j+k-n-1\}} \leq r \leq \min{\{j,k\}}-1.\]
\end{lem}

\proof
Since our formula is symmetric in $j$ and $k$ except for the upper limit of summation,
proving the lemma when $j \leq k$ will suffice. When $j \leq k$ the upper limit is $j-1$,
rather than $k-1$ when $j > k$. $Q_n(j,k)$ represents the number of permutations
$\sigma \in \CS_n(132)$ with $\sigma(j)=k$. Let $q_n(j,k,r)$ be the number of
\textbf{132}-avoiding permutations $\sigma$ counted by $Q_n(j,k)$ such that there
are exactly $r$ values $x$ with $x<j$ such that $\sigma(x) < k$.
Below we show that
$$
q_n(j,k,r) \. = \. b(n-j+1,k-r) \. b(n-k+1,j-r) \. C_r \quad \text{for all} \ \, 0 \leq r \leq j-1,
$$
which implies the result.

Let $\sigma \in \CS_n(132)$ such that $\sigma(j) = k$ and there are exactly $r$ numbers~$x_i$
with $x_i < j$ and $\sigma(x_i) < \sigma(j)=k$. We use decomposition \.
$\sigma \ts = \ts \tau \ts \pi \ts k \ts \phi$, where
$$
\tau = \{\sigma(1),\ldots,\sigma(j-r-1)\}, \ \, \pi = \{\sigma(j-r),\ldots,\sigma(j-1)\}, \ \,
\text{and} \ \. \phi=\{\sigma(j+1),\ldots,\sigma(n)\}\..
$$
Observe that either all elements of $\pi$ are smaller than $k$, or there is some element of~$\pi$
greater than~$k$, and some element of $\tau$ smaller than~$k$. Suppose the second case is true,
with $a$ an element of~$\tau$ smaller than~$k$, and $b$ an element of $\pi$ larger than $k$.
Then $a,b$, and~$k$ form a \textbf{132}-pattern, a contradiction.  Therefore, all elements of~$\pi$
must be smaller than~$k$.

Suppose some element $x$ of $\pi$ is smaller than $(k-r)$. Then some number $y$ with
$k-r \leq y < k$ is an element of $\phi$. Then we have a \textbf{132}-pattern,
formed by $x,k$, and $y$, which is a contradiction, so $\pi$ consists of
$\{k-r,k-r+1,\ldots,k-2,k-1\}$. There are $C_r$ possible choices for $\pi$,
since $\pi$ must avoid the \textbf{132}-pattern.

Now consider the values of $\sigma$ within $\tau$. Observe that regardless of~$\tau$,
the numbers $s$ in $\phi$ with $s>k$ must be in decreasing order in~$\phi$,
in order to avoid a \textbf{132}-pattern that starts with$~k$. Therefore,
the number of possible choices for $\tau$ is equal to the number of possible orderings
of the numbers between $k$ and $n$ that avoid~\textbf{132}, with $k$ in the
$(j-r)$-th position (since $\pi$ only consists of numbers smaller than~$k$).
There are exactly $b(n-k+1,j-r)$ such possible choices.

Finally, consider values of~$\sigma$ within $\{k\} \cup \phi$. Here we need to order
$n-j+1$ numbers so that they avoid the \textbf{132}-pattern, with the first number
being the $(k-r)$-th smallest. There are exactly $b(n-j+1,k-r)$ ways to do this.
Choosing $\pi$, $\tau$, and $\phi$ completely determines~$\sigma$. Therefore,
there are $b(n-j+1,k-r)\,b(n-k+1,j-r)\,C_r$ possible choices for~$\sigma$,
as desired.
\qed
\medskip
\begin{Example}\label{exP}
\rm
Let us compute $Q_7(4,3)$, the number of permutations $\sigma \in \CS_n(132)$
with $\sigma(4)=3$. We count the permutations separately depending on how many
numbers smaller than 3 come ahead of 3. First suppose $r=0$, so there are no
numbers smaller than 3 ahead of 3 in the permutation. This means that~3 is in
the 4-th position among those numbers greater than equal to 3. There are
\[b(7-3+1,4-0) = \frac{5-4+1}{5+4-1}\binom{5+4-1}{5} = \frac{2}{8}\binom{8}{5} =14\]
ways to order the numbers between 3 and 7, displayed here:
\[\begin{array}{ccccccc}
(45637)&(54637)&(56437)&(56734)&(64537)&(65437)&(65734)\\
(67435)&(67534)&(74536)&(75436)&(75634)&(76435)&(76534).\\
\end{array}\]
For each of these there are
\[b(7-4+1,3-0)= \frac{2}{6}\binom{6}{4} =5\]
ways to place the numbers between 1 and~3, shown here:
\[\begin{array}{ccccc}
(***\,3\,1\,2\,*)&(***\,3\,2\,1\,*)&(***\,3\,2*1)&(***\,3*1\,2)&(***\,3*2\,1),\\
\end{array}\]
where the stars represent the positions of 4,5,6 and 7.
In total we find that $q_7(4,3,0)=b(5,4)b(4,3)=(14)(5)=70$.

Similarly, for $r=1$ there is 1 number smaller than 3 ahead of 3 in the permutation.
This number has to be~2,  since otherwise a \textbf{132} pattern would be formed with
1,2, and~3. Also this number must be directly in front of 3 in the permutation,
since otherwise a \textbf{132} would be formed with 2 as the 1 and 3 as the 2.
This means that 3 is now in the 3-rd position among those numbers greater than
equal to~3. There are
\[b(7-3+1,4-1) = \frac{3}{7}\binom{7}{5} =9\]
ways to order the numbers between 3 and~7, displayed here:
\[\begin{array}{ccccccccc}
(45367)&(54367)&(56347)&(64357)&(65347)&(67345)&(74356)&(75346)&(76345).\\
\end{array}\]
For each of those orderings there are
\[b(7-4+1,3-1)C_1 = \frac{3}{5}\binom{5}{4} =3\]
ways to place the numbers between 1 and 3, displayed here:
\[\begin{array}{ccc}
(**2\,3\,1**)&(**2\,3*1*)&(**2\,3**1).\\
\end{array}\]
Therefore, overall we have $q_7(4,3,1)=b(5,3)b(4,2)C_1=(9)(3)=27$.
In other words, there are 27 distinct \textbf{132}-avoiding permutations of length~7
with 3 in the 4-th position, and one number smaller than~3 and ahead of~3.

The final case is when $r=2$, in which case both 1 and 2 come ahead
of~3 in the permutation. This means that 3 is in the 2-nd position
among those numbers between 3 and 7. Therefore there are
\[b(7-3+1,4-2)=\frac{4}{6}\binom{6}{5}=4\]
ways to order the numbers between 3 and~7,
displayed here:
\[\begin{array}{cccc}
(43567)&(53467)&(63457)&(73456).\\
\end{array}\]
For each of these we have \[b(7-4+1,3-2)C_2=\binom{4}{4}(2)=2\]
ways to place the numbers between 1 and 3, displayed here:
\[\begin{array}{cc}
(*1\,2\,3***)&(*\,2\,1\,3***).\\
\end{array}\]
Together we have that $q_7(4,3,2)=b(5,2)b(4,1)C_2=(4)(1)(2)=8$, so there are 8
different \textbf{132}-avoiding permutations of length~7 with 3 in the 4-th position,
and two numbers smaller than~3 and ahead of~3.
We have shown that
\[Q_7(4,3)=q_7(4,3,0)+q_7(4,3,1)+q_7(4,3,2)=70+27+8=105.\]
\end{Example}

\medskip
\subsection{Proof of theorems~\ref{T3}, \ref{T6}, and~\ref{T4}}
The proof again involves one technical lemma and several cases corresponding to the statements
of Theorems~\ref{T3} and~\ref{T4}.

Let $h:[0,1]^3 \to \mathbb{R}$ be defined so that
\[h(a,s,t) = \frac{4^{ast}(1-at+a-ast)^{1-at+a-ast}(1-a+at-ast)^{1-at+a-ast}}{(1-at)^{(1-at)}(a-ast)^{(a-ast)}(1-a)^{(1-a)}(at-ast)^{(at-ast)}}.\]

\begin{lem}\label{Pbase}
For all $(a,s,t) \in [0,1]^3$, we have $h(a,s,t) \leq 4$.
Moreover,
\[h(a,s,t) = 4 \,\,\,\text{ if and only if }\,\, s=\frac{at+a-1}{at}.\]
\end{lem}
\proof
Take the logarithmic derivative of $h$ to obtain
\begin{align*}
\frac{d(\ln{h})}{ds}
&= at\ln{4}+(-at(1+\ln{(1-at+a-ast)}))+(-at(1+\ln{(1-a+at-ast)}))  \\ & \qquad \quad -\left[-at(1+\ln{(a-ast)})-at(1+\ln{(at-ast)})\right]\\
&= at \ln{\left[4(a-ast)(at-ast)\right]}- at\ln{\left[(1-at+a-ast)(1-a+at-ast)\right]}.
\end{align*}
Set this derivative equal to 0 to get
$$4(a-ast)(at-ast) - (1-at+a-ast)(1-a+at-ast)=0\ts,
$$
or
$$(ast-(at+a-1))\ts (3ast-(at+a+1))\. = \. 0\ts,$$
giving
$$
s \. = \.\frac{at+a-1}{at}\quad  \text{ or }\quad s \. = \. \frac{at+a+1}{3at}\..
$$
Since
$$
\frac{at+a+1}{3at} \, = \,  \frac{1}{3}\. +\. \frac{1}{3t}\. +\. \frac{1}{3at}
\, \geq \, \frac{1}{3}+\frac{1}{3}+\frac{1}{3}=1\ts,
$$

\smallskip

\noindent
this value of $s$ is greater than 1, and is only equal to 1 if $a=s=t=1$.
Similarly, the ratio \. $(at+a-1)/at$ \. is  between 0 and~1 if $at+a>1$.
It is easy to see that the second derivative
\[\frac{d^2(\ln{h})}{(ds)^2} < 0 \ \,\text{ at }\,\,s=\frac{at+a-1}{at}\.,
\]
which implies that this value of $s$ does indeed maximize~$h(a,s,t)$.
We can also verify that \[h\Big(a,\frac{at+a-1}{at},t\Big) \.
=\. \frac{4^{(at+a-1)}(2-2at)^{(2-2at)}(2-2a)^{(2-2a)}}{(1-at)^{(1-at)}(1-at)^{(1-at)}(1-a)^{(1-a)}(1-a)^{(1-a)}}\. =\. 4\ts.
\]
Observe that $h(a,s,t) < 4$ on the boundary of $[0,1]^3$ except for where $a=s=t=1$, completing the proof.
\qed

\medskip

\begin{lem}\label{LessThan1}
Let $a,b \in [0,1]^2, c \neq 0$ and
$0 \leq \alpha <1$, such that $a+b<1$. Then $G(a,b,c,\alpha) = \infty$.
Moreover, for $n$ sufficiently large, we have
\[\frac{Q_n(an-cn^{\alpha},bn-cn^{\alpha})}{C_n} < \ve^n,\]
where~$\ve$ is independent of~$n$, and $0<\ve<1$.
\end{lem}
\proof
By Lemma~\ref{explicitQ}, we have
\[ Q_n(atn,an) = \sum_r \,b(n-atn+1,an-r) \,b(n-an+1,atn-r)\,C_r \]
\[= \sum_r \,\left[\frac{n-atn+1-(an-r)+1}{n-atn+1+(an-r)-1}\binom{n-atn+an-r}{n-atn+1}\right]\]
\[\qquad\qquad\qquad\times\left[\frac{n-an+1-atn+r+1}{n-an+1+atn-r-1}\binom{n-an+1+atn-r-1}{n-an+1}\right] \]
\[ \times \left[\frac{1}{r+1}\binom{2r}{r}\right],\]
where the summation is over values of $r$ such that
\[\max{\{0,j+k-n-1\}} \. \leq \. r \. \leq \. \min{\{j,k\}}-1\ts.\]
Let $r = atsn$, so $s$ varies from 0 to $\left(1-\frac{1}{atn}\right)$ by increments of $\frac{1}{atn}$.
Applying Stirling's formula, we get
\[Q_n(atn,an) \sim \sum_r\,\,\chi(n,atn,an,astn)\,\,h(a,s,t)^n,\]
where
$$\chi(n,atn,an,astn) = \sqrt{\frac{(1-at+a-ats)(1-a+at-ats)}{(1-at)(a-ats)(1-a)(at-ats)(ats)}}
$$
$$
\times \frac{a^2t(1-s)(1-st)(n(1-a-at+ats)+2)^2}{2(\pi n)^{3/2}(n(1-at)+1)(n(1-a)+1)(1+a-at-ats)(1-a+at-ats)(atsn+1)}\ts.
$$

\smallskip

\noindent
We now have
$$\frac{n^dQ_n(an-cn^{\alpha},bn-cn^{\alpha})}{C_n} \sim \sqrt{\pi}n^{d+\frac{3}{2}}\sum_{r=0}^{an-cn^{\alpha}-1}\nu_r(n)\ts,$$
where
$$\nu_r(n) = \chi(n,an-cn^{\alpha},bn-cn^{\alpha},r)\,h(b,r/an,a/b)^n\. 4^{-n}.
$$

\smallskip

\noindent
From Lemma~\ref{Pbase}, we have that $h(b,r/an,a/b)<4$ for $r \neq (a+b-1)n$.
For values of $r$ where $h(b,r/an,a/b) < 4$, $\nu_r(n)$ decreases exponentially as
$n \to \infty$ for fixed~$d$. Therefore, for these values of~$r$,
\[\sqrt{\pi}n^{d+\frac{3}{2}}\nu_r(n) \to 0\,\,\,\text{as}\,\,n \to \infty\,\,\,\text{for all}\,\,d\ts.\]
The only values of $r$ which could potentially have $\lim_{n \to \infty} \nu_r(n) \neq 0$,
are when $r \sim (a+b-1)n$, as $n \to \infty$.
Observe that since $a+b-1<0$, there are no such possible values of~$r$. In this case,
\[\lim_{n \to \infty} \frac{n^dQ_n(an-cn^{\alpha},bn-cn^{\alpha})}{C_n}=0 \,\,\,\text{ for all }\,\,d > 0.\]
This implies $G(a,b,c,\alpha)=\infty$ when $a+b<1$. Also, for $n$ large enough, we have
\[\frac{Q_n(an-cn^{\alpha},bn-cn^{\alpha})}{C_n} < \left(\frac{1+h(b,0,a/b)}{2}\right)^n,\]
as desired.
\qed

\medskip
\begin{lem}\label{AlphaBig} Let $a \in [0,1],c>0$, and $\frac{1}{2} < \alpha< 1$.
Then $G(a,1-a,c,\alpha)=\infty.$ Moreover, for $n$ large enough, we have
\[\frac{Q_n(an-cn^{\alpha},(1-a)n-cn^{\alpha})}{C_n} < \ve^{n^{2\alpha-1}},\]
where $\ve = \ve(a,c,\alpha)$ is independent of $n$, and $0<\ve<1$.
\end{lem}
\proof
Let $k = cn^{\alpha}$. Then
\[Q_n(an-k,(1-a)n-k) = q_n(an-k,(1-a)n-k,0)\,\,S_{a,n,k}\,\,,\]
where
\[S_{a,n,k}= \left(1+\sum_{r=1}^{an-k-1}g_r(n)\right),\]
and
\[g_r(n) = \left(\frac{q_n(an-k,(1-a)n-k,r)}{q_n(an-k,(1-a)n-k,0)}\right).\]
Observe that
\[q_n(an-k,(1-a)n-k,0)=P_n(an-k,(1-a)n-k).\]
Applying Stirling's formula and using the Taylor expansion for $\ln(1+x)$ gives
\[g_r(n) \sim \frac{(2k+r+2)^2C_r}{(2k+2)^24^r}\,\,
\exp{\left[{\frac{-(4kr+r^2)}{4a(1-a)n}}\right]}.\]
Therefore,
\[ \frac{n^dQ_n(an-cn^{\alpha},bn-cn^{\alpha})}{C_n}\sim
\left(\frac{n^dP_n(an-cn^{\alpha},bn-cn^{\alpha})}{C_n}\right) \]
\[ \times \left(1+\sum_{r=1}^{an-k-1}\frac{(2k+r+2)^2C_r}{(2k+2)^24^r}\,\,
\exp{\left[{\frac{-(4kr+r^2)}{4a(1-a)n}}\right]}\right).\]
For $\alpha > \frac{1}{2}$, by Theorem~\ref{T1}, we have
\[\lim_{n \to \infty}\frac{n^dP_n(an-cn^{\alpha},(1-a)n-cn^{\alpha})}{C_n}=0\text{ for all }d\ts.\]
Therefore, $G(a,1-a,c,\alpha)=\infty$ for $\alpha > \frac{1}{2}$. Also, we have proven the second
case of Theorem~\ref{T6}, as desired.
\qed

\medskip
For the next three cases, we denote $r = h\ts n^p$, and let $h$ and $p$ be fixed as $n \to \infty$.
For $p>\frac{1}{2}$, we have
\[\frac{n^dg_r(n)}{C_n} \to 0\,\,\text{ as }\,\,n \to \infty,\,\,\text{ for all }\,\,d\ts,\]
since $g_r(n)$ decreases exponentially for fixed~$d$.

For $\alpha <p< \frac{1}{2}$, we have
\[g_r \sim \frac{r^{\frac{1}{2}}}{4\sqrt{\pi}k^2} \sim
\left(\frac{\sqrt{h}}{4\sqrt{\pi}}\right)n^{\frac{p}{2}-2\alpha}.\]
For \. $p \leq \alpha < \frac{1}{2}$ \ts or \ts $p < \alpha =\frac{1}{2}$\ts,
we have
\[g_r = \Theta\left(r^{-\frac{3}{2}}\right) = \Theta\left(n^{-\frac{3p}{2}}\right).\]
Similarly, for $p=\frac{1}{2}$, we obtain
\[g_r =\Theta\left( n^{\frac{1}{4}-2\alpha}\right).\]

\begin{lem}\label{SmallAlphaQ}
Let $a \in [0,1], c\in \rr$ and $0 \leq \alpha < \frac{3}{8}$. Then
\[G(a,1-a,c,\alpha) = \frac{3}{4} \quad\text{ and }\quad M(a,1-a,c,\alpha)=z(a),\]
where \[z(a) = \frac{\Gamma(\frac{3}{4})}{2^{\frac{9}{4}}\pi\left[a(1-a)\right]^{\frac{3}{4}}}\,\,\,
\text{as in Theorem~\ref{T4}}.\]
\end{lem}
\proof

As in Lemma~\ref{AlphaBig}, we write
\[Q_n(an-cn^{\alpha},(1-a)n-cn^{\alpha}) = P_n(an-cn^{\alpha},(1-a)n-cn^{\alpha})\,S_{a,n,k}\,,\]
where
\[S_{a,n,k} = 1+\sum_{r=0}^{an-k-1}g_r(n).\]
Again, as $n \to \infty$, we have
\[g_r(n) \sim \frac{(2k+r+2)^2C_r}{(2k+2)^24^r}\,\,
\exp{\left[{\frac{-(4kr+r^2)}{4a(1-a)n}}\right]}.\]
Fix $s>0$, and observe that for any $0 \leq \alpha < \frac{3}{8}$,
we have
\[g_{n^\delta}(n) = o(g_{s\sqrt{n}}(n)),\,\,\text{ for every }\,\,\delta \neq \frac{1}{2}.\]
Therefore, as $n \to \infty, t \to 0$, and $u \to \infty$, we have
\[S_{a,n,k} \sim \sum_{r=t\sqrt{n}}^{u\sqrt{n}}g_r(n).\]

Interpreting this sum as a Riemann sum, we have
\[S_{a,n,k} \sim \sqrt{n}\int_t^ug_{v\sqrt{n}}(n)dv  \]
\[ \sim \sqrt{n}\int_t^u \frac{(2k+v\sqrt{n}+2)^2}{(2k+2)^2\sqrt{\pi}(v\sqrt{n})^{\frac{3}{2}}
}\left(\exp{\left[\frac{-(4kv\sqrt{n}+(v\sqrt{n})^2)}{4a(1-a)n}\right]}\right) dv.\]
Therefore, we have
\[S_{a,n,k} \sim  \sqrt{n}\int_t^u \frac{v^{\frac{1}{2}}n^{\frac{1}{4}}}{4k^2\sqrt{\pi}}
\left(\exp{\left[\frac{-v^2}{4a(1-a)}\right]}\right) dv.\]
A direct calculation gives
\[S_{a,n,k} = \frac{n^{\frac{3}{4}-2\alpha}}{c^2}\,\,z(a)
\left(\sqrt{\pi}\left[a(1-a)\right]^{\frac{3}{2}}\right).\]

Now we see that
\[\frac{n^dQ_n(an-cn^{\alpha},(1-a)n-cn^{\alpha})}{C_n} \sim
\frac{n^dP_n(an-cn^{\alpha},(1-a)n-cn^{\alpha})}{C_n}S_{a,n,k}\]\[ \sim z(a)\,n^{d-\frac{3}{4}},\]
by the proof of Theorem~\ref{T1} and the analysis of $S_{a,n,k}$. Therefore,
$G(a,1-a,c,\alpha) = \frac{3}{4}$. For $\alpha < \frac{3}{8}$
this also gives $M(a,1-a,c,\alpha) = z(a)$, as desired.
\qed

\medskip
This case displays why we do not need to adjust our analysis to be on the anti-diagonal.
Since the behavior of~$Q$ depends on values of $r$ on the order of $\sqrt{n}$, adding 1 to
the second coordinate is a lower-order term and does not affect $G$ or $M$ at all. In fact,
the whole value of $cn^{\alpha}$ has no effect on $G$ or $M$ for this case.

\begin{lem}\label{Over38ths}
Let $a \in [0,1],c>0$, and $\frac{3}{8} < \alpha \leq \frac{1}{2}$\ts.
Then
\[G(a,1-a,c,\alpha) = \frac{3}{2}-2\alpha\,.\]
Moreover,
\[M(a,1-a,c,\alpha) = \begin{cases} y(a,c) & \text{ if }\,\,\frac{3}{8} < \alpha < \frac{1}{2}\,,\\
y(a,c)\,\kappa(a,c) & \text{ if }\,\,\alpha = \frac{1}{2}\,,\\
\end{cases}\]
where $y(a,c)$ and $\kappa(a,c)$ are defined as in Theorem~\ref{T2}.
\end{lem}

\proof
As in the previous lemma, we have $
Q_n(an-cn^{\alpha},(1-a)n-cn^{\alpha}) = P_n(an-cn^{\alpha},(1-a)n-cn^{\alpha})\,S_{a,n,k}$.
We analyze $S_{a,n,k}$ to see which values of $r$ contribute the most. In this case, for
$p =\frac{1}{2}$, we have $g_r = \Theta(n^{\frac{1}{4}-2\alpha})$,
which is on the order of $n^d$ with $d$ strictly less than $-\frac{1}{2}$\ts.
Therefore, even if we sum over all values of $r$ where $p = \frac{1}{2}$,
we will end up with an expression on the order of $n^{d+\frac{1}{2}}$ which is
lower order than a constant. For $p \leq \alpha < \frac{1}{2}$ or
$p < \alpha = \frac{1}{2},$ since $g_r = \Theta(n^{-\frac{3p}{2}})$,
the terms with the highest order will come when $p=0$. Therefore,
the values of $r$ which contribute the most to $S_{a,n,k}$ will be
constants in this case.  From this, we have
\[S_{a,n,k} \sim 1+\sum_{r=1}^s\, g_r \sim 1+\sum_{r=1}^s\, \frac{C_r}{4^r} \sim
2,\,\,\text{ as }\,n \to \infty.\]
Therefore,
\[\frac{n^dQ_n(an-k,(1-a)n-k)}{C_n} \sim
\frac{2n^dP_n(an-k,(1-a)n-k)}{C_n}.\]
Referring back to Theorems~\ref{T1} and~\ref{T2} gives us the desired results.
\qed

\medskip
\begin{lem}\label{3/8}
Let $a \in [0,1]$ and $c>0$. Then $G(a,1-a,c,3/8) = \frac{3}{4}$
and $M(a,1-a,c,3/8) = z(a)+y(a,c)$.
\end{lem}
\proof
Here we are essentially on the intersection of the last two cases,
which provides some intuition to the reason that
$M(a,1-a,c,3/8) = z(a)+y(a,c)$.  Again we let $r = hn^p$.
When $\alpha = \frac{3}{8}$, values of $g_r$ which contribute the highest order to
$S_{a,n,k}$ are when $p=0$ and when $p=\frac{1}{2}$\ts~. We get $z(a)$ from the terms where
$p=\frac{1}{2}$\ts, and $y(a,c)$ when $p=0$.
\[S_{a,n,k} \sim \frac{z(a)}{c^2}\,\,\left(\sqrt{\pi}\left[a(1-a)\right]^{\frac{3}{2}}\right)
+2,\]
so
\[\frac{n^dQ_n(an-k,(1-a)n-k)}{C_n} \sim
\frac{2c^2}{\sqrt{\pi}(a(1-a))^{\frac{3}{2}}}n^{d-\frac{3}{4}} + z(a)n^{d-\frac{3}{4}},\]
so
$G(a,1-a,c,\alpha) = \frac{3}{4}$ and $M(a,1-a,c,\alpha)=z(a)+y(a,c)$,
as desired.
\qed

\medskip

\begin{lem}\label{corner}
Let $c>0$ and $0 \leq \alpha < 1$. Then \[G(1,1,c,\alpha)=\frac{3}{2}\alpha.\] Moreover, for $\alpha>0$, we have $M(1,1,c,\alpha) = w(c)$, and $M(1,1,c,0) = u(c)$,
where $w(c)$ and $u(c)$ are defined as in Theorem~\ref{T4}.
\end{lem}

\proof
We first consider $Q_n(n-k,n-k)$ with $k=cn^{\alpha}$, and $\alpha>0$. We have
\[Q_n(n-k,n-k) = \sum_{i=n-2k-1}^{n-k-1}b(k+1,n-k-i)^2C_i\,\,,\]
which is equivalent to
\[Q_n(n-k,n-k) = C_kC_kC_{n-2k-1}\left(1+\sum_{r=1}^k h_r(n)\right),\]
where
\[h_r(n) = \frac{(r+1)^2\binom{2k-r}{k}^2}{\binom{2k}{k}^2}\left(\frac{C_{n-2k-1+r}}{C_{n-2k-1}}\right).\]
Observe that
$h_r(n) \sim r^2\exp{\left[\frac{-r^2}{2k}\right]}$, so
$h_{k^{\delta}}(n) = o(h_{k^{\frac{1}{2}}}(n))$ for all
$\delta \neq \frac{1}{2}$.
From here we have
\[\frac{n^dQ_n(n-k,n-k)}{C_n} \sim\frac{n^d}{4\pi k^3}\int_0^{\infty}n^{\frac{3}{2}\alpha}s^2\exp{\left[\frac{-s^2}{2c}\right]}ds\]
\[\sim \frac{n^{d-\frac{3}{2}\alpha}}{4\pi c^3}\int_0^{\infty}s^2\exp{\left[\frac{-s^2}{2c}\right]}ds.\]
Therefore $G(1,1,c,\alpha) = \frac{3}{2}\alpha$ and \[M(1,1,c,\alpha) = \frac{1}{2^{\frac{5}{2}}\pi^{\frac{1}{2}}c^{\frac{3}{2}}}=w(c),\]
as desired.

For $\alpha=0$, \[Q_n(n-k,n-k)=Q_n(n-c,n-c) = \sum_{i=n-2c-1}^{n-c-1}b(c+1,n-c-i)^2C_i,\]
\[=\sum_{s=0}^c\left(\frac{s+1}{2c+1-s}\binom{2c+1-s}{c+1}\right)^2C_{n-2c-1+s}.\]
Therefore, we have \[\frac{Q_n(n-k,n-k)}{C_n} \sim \frac{1}{4^{2c+1}}\sum_{s=0}^c\left(\frac{s+1}{2c+1-s}\binom{2c+1-s}{c+1}\right)^24^s\,\,\text{ as }n \to \infty,\] completing the proof.
\qed

\medskip
\begin{lem}\label{CPlus}
Let $a \in [0,1], c<0$ and $0 < \alpha < 1$. Then \[G(a,1-a,c,\alpha) = \begin{cases}
\frac{3}{4} & \text{ if }\,\,0 < \alpha \leq \frac{1}{2}\,, \\
\frac{3}{2}\alpha & \text{ if }\,\,\frac{1}{2} < \alpha < 1\,,\\
\end{cases}\]
and \[M(a,1-a,c,\alpha) = \begin{cases}
z(a) & \text{ if }\,\,0 < \alpha < \frac{1}{2}\,,\\
x(a,c) & \text{ if }\,\,\alpha = \frac{1}{2}\,,\\
w(c) & \text{ if }\,\,\frac{1}{2} < \alpha < 1\,,\\
\end{cases}\]
where $z(a), x(a,c)$, and $w(c)$ are defined as in Theorem$~\ref{T4}$.
\end{lem}
\proof
We now analyze $Q_n(an+k,(1-a)n+k)$, where $k = cn^\alpha$. We have
\[Q_n(an+k,(1-a)n+k) = \sum_{r=2k-1}^{an+k-1}\,b((1-a)n-k+1,(1-a)n+k-r)\,\,b(an-k+1,an+k-r)\,C_r\,.\]
We can rewrite this as
\[Q_n(an+k,an+k)=\sum_{d=0}^{an-k}\,b((1-a)n-k+1,(1-a)n-k-d+1)\,\,b(an-k+1,an-k-d+1)\,C_{d+2k-1}\,.\]
Denote by $g_d$ the $d$-th term of this summation. Then we can express $Q_n(an+k,an+k)$ as
\[Q_n(an+k,an+k)=g_0\left(1+\sum_{d=1}^{an-k}\frac{g_d}{g_0}\right).\]
Denote $h_d$ as $h_d = \frac{g_d}{g_0}$. We can express $h_d$ as
\[h_d = \frac{(d+1)^2\binom{2(an-k)-d}{an-k}\binom{2((1-a)n-k)-d}{(1-a)n-k}}{\binom{2(an-k)}{an-k}
\binom{2((1-a)n-k)}{(1-a)n-k}}\left(\frac{C_{d+2k-1}}{C_{2k-1}}\right).\]
We now have
\[h_d \sim d^2\exp{\left[\frac{-d^2}{4an}\right]}\exp{\left[\frac{-d^2}{4(1-a)n}\right]}
\left(\frac{C_{d+2k-1}}{4^dC_{2k-1}}\right).\]
Fix $s > 0$ and $0 \leq \beta <1$. We set $d = sn^{\beta}$, and consider the behavior of $h_d$
as $n \to \infty$. Observe that for any~$\alpha$,
\[h_{n^{\delta}} = o(h_{sn^{\frac{1}{2}}})\]
when $\delta \neq \frac{1}{2}$.
For $\frac{1}{2}=\beta<\alpha$, we have
\[h_d \sim d^2\exp{\left[\frac{-s^2}{4a(1-a)}\right]}.\]
For $\frac{1}{2}=\beta=\alpha$, we have
\[h_d \sim d^2\exp{\left[\frac{-s^2}{4a(1-a)}\right]}\left(\frac{2c}{s+2c}\right)^{\frac{3}{2}}.\]
Similarly, for $\alpha < \beta = \frac{1}{2}$,
\[h_d \sim d^\frac{1}{2}(2k)^{\frac{3}{2}}\exp{\left[\frac{-s^2}{4a(1-a)}\right]}.\]
Therefore, as $n \to \infty$, $t \to 0$, and $u \to \infty$,
\[Q_n(an+k,an+k) \sim g_0\left(1+\sum_{d=tn^{\frac{1}{2}}}^{un^{\frac{1}{2}}}h_d\right).\]
Recall that
\[g_0 \sim \frac{4^{n-1}}{(2\pi a(1-a)k)^{\frac{3}{2}}n^3}\,.\]
We are now ready to analyze $G(a,1-a,c,\alpha)$. For
$\alpha < \frac{1}{2}$ we have
\[\frac{n^dQ_n(an+cn^{\alpha},(1-a)n+cn^{\alpha})}{C_n} \sim
\frac{n^d}{4\pi(2a(1-a)c)^{\frac{3}{2}}n^{\frac{3}{2}(1+\alpha)}}
\int_0^{\infty}n^{\frac{3}{4}}(2k)^{\frac{3}{2}}s^{\frac{1}{2}}
\exp{\left[\frac{-s^2}{4a(1-a)}\right]}ds\]
\[\sim\frac{n^{d-\frac{3}{4}}}{4\pi(a(1-a))^{\frac{3}{2}}}
\left(\frac{\Gamma(\frac{3}{4})(4a(1-a))^{\frac{3}{4}}}{2}\right), \]
so $G(a,1-a,c,\alpha)=\frac{3}{4}$ and $M(a,1-a,c,\alpha) = z(a)$, as desired.
For $\alpha>\frac{1}{2}$ we have
\[\frac{n^dQ_n(an+cn^{\alpha},(1-a)n+cn^{\alpha})}{C_n} \sim
\frac{n^d}{4\pi(2a(1-a)c)^{\frac{3}{2}}n^{\frac{3}{2}(1+\alpha)}}
\int_0^{\infty}n^{\frac{3}{2}}s^2\exp{\left[\frac{-s^2}{4a(1-a)}\right]}ds\]
\[\sim\frac{n^{d-\frac{3}{2}\alpha}}{4\pi(2a(1-a)c)^{\frac{3}{2}}}
\left(\frac{\pi^{\frac{1}{2}}(4a(1-a))^{\frac{3}{2}}}{4}\right),\]
so $G(a,1-a,c,\alpha) = \frac{3}{2}\alpha$ and $M(a,1-a,c,\alpha) = w(c)$.
For $\alpha = \frac{1}{2}$ we have \[\frac{n^dQ_n(an+cn^{\alpha},(1-a)n+cn^{\alpha})}{C_n}
\sim \frac{n^d}{4\pi(2a(1-a)c)^{\frac{3}{2}}n^{\frac{3}{2}(1+\alpha)}}
\int_0^{\infty}n^{\frac{3}{2}}s^{2}\exp{\left[\frac{-s^2}{4a(1-a)}\right]}
\left(\frac{2c}{s+2c}\right)^{\frac{3}{2}}ds\]
\[\sim\frac{n^{d-\frac{3}{4}}}{4\pi(a(1-a))^{\frac{3}{2}}}
\int_0^{\infty}\frac{s^2}{(s+2c)^{\frac{3}{2}}}
\exp{\left[\frac{-s^2}{4a(1-a)}\right]}ds \sim n^{d-\frac{3}{4}}x(a,c),\]
proving that $G(a,1-a,c,\frac{1}{2})=\frac{3}{4}$ and $M(a,1-a,c,\frac{1}{2}) = x(a,c)$.
This completes the proof.
\qed

\medskip

Our final case is where $1 < a+b < 2$.

\begin{lem}\label{CPlusGen}
Let $a,b \in [0,1]$, such that $1<a+b<2$, and let~$c \in \rr$, $0 < \alpha < 1$. Then
\[G(a,b,c,\alpha) = \frac{3}{2} \,\,\,\text{ and } \,\,M(a,b,c,\alpha) = v(a,b),\]
where $v(a,b)$ is defined as in Theorem~\ref{T4}.
\end{lem}

\proof
Let $k=cn^{\alpha}$. We analyze $Q_n(an+cn^{\alpha},bn+cn^{\alpha})$
in the same manner as the proof of Lemma~\ref{CPlus}, by looking at $Q_n(an+k,bn+k)$.
We obtain:
\[Q_n(an+k,bn+k) = C_{(1-a)n+k}C_{(1-b)n+k}C_{(a+b-1)n-2k-1}\left(1+\sum_{d=1}^{(1-b)n}g_d\right),\]
where
\[g_d = \frac{(d+1)^2\binom{2((1-a)n+k)-d}{(1-a)n+k}\binom{2((1-b)n+k)-d}{(1-b)n+k}}
{\binom{2((1-a)n+k)}{(1-a)n+k}\binom{2((1-b)n+k)}{(1-b)n+k}}\left(\frac{C_{(a+b-1)n-1+d}}{C_{(a+b-1)n-1}}\right).\]
We find that
\[g_d \sim d^2 \exp{\left[\frac{-d^2(2-a-b)}{4(1-a)(1-b)n}\right]},\]
so $g_{n^{\beta}} = o(g_{n^{\frac{1}{2}}})$ if $\beta \neq \frac{1}{2}$.
Therefore,
\[\frac{n^dQ_n(an+cn^{\alpha},bn+cn^{\alpha})}{C_n} \sim
\frac{n^{d-3}}{4\pi((1-a)(1-b)(a+b-1))^{\frac{3}{2}}}\int_0^{\infty}n^{\frac{3}{2}}s^2
\exp{\left[\frac{-s^2(2-a-b)}{4(1-a)(1-b)}\right]}ds\,.\]
We can now see that $G(a,b,c,\alpha) = \frac{3}{2}$ and $M(a,b,c,\alpha) = v(a,b)$,
as desired. We have now proved all cases of Theorems~\ref{T3} and~\ref{T4}.
\qed

\bigskip

\section{Expectation of basic permutation statistics}\label{s:exp}


\subsection{Results}
The following result describes the behavior of the first and last elements
of~$\sigma$ and~$\tau$.

\begin{thm}\label{head.last}
Let $\sigma \in \CS_n(123)$ and $\tau \in \CS_n(132)$ be permutations chosen
uniformly at random from the corresponding sets. Then
\begin{equation*}
\eE[\sigma(1)] = \eE[\sigma^{-1}(1)] = \eE[\tau(1)] = \eE[\tau^{-1}(1)] \to n-2\,\,
\text{ as }n \to \infty,\tag{1}\label{case1}
\end{equation*}
\begin{equation*}
\eE[\sigma(n)] = \eE[\sigma^{-1}(n)] \to 3,\,\,\text{ as }n \to \infty,
\tag{2}\label{case2}
\end{equation*}
and
\begin{equation*}
\eE[\tau(n)] = \eE[\tau^{-1}(n)] = \frac{(n+1)}{2}\,\,\text{ for all }n\ts.
\tag{3}\label{case3}
\end{equation*}
\end{thm}

We remark here that the above theorem can be proved by using the exact formulas for $P_n(j,k)$ and $Q_n(j,k)$ shown in Lemmas~\ref{explicitP} and \ref{explicitQ}. We will instead prove the theorem by analyzing bijections between $S_n(123)$ and $S_n(132)$.  Before proving Theorem~\ref{head.last} we need some notation and definitions.

\medskip

\subsection{Definitions}
Let $A,B$ be two finite sets. We say $A$ and $B$ are \emph{equinumerous} if
$\vert A \vert = \vert B \vert$. Let $\alpha:A \to \zz$ and $\beta:B \to \zz$.
We say $\alpha$ and $\beta$ are \emph{statistics} on $A$ and $B$, respectively.
We say that two statistics $\alpha$ and $\beta$ on equinumerous sets are
\emph{equidistributed} if $\vert \alpha^{-1}(k) \vert = \vert \beta^{-1}(k) \vert$
for all $k \in \zz$; in this case we write $\alpha \sim \beta$.
Note that being equidistributed is an equivalence relation.

For example, let $A = \CS_n(132)$, let $\alpha:A \to \zz$,
such that $\alpha(\sigma) = \sigma(1)$,
and let $\beta:A \to \zz$ such that $\beta(\sigma) = \sigma^{-1}(1)$.
Then $\alpha \sim \beta$, by Proposition~\ref{Symmetry}. \ts
Given $\pi \in S_n$, let
\[\rmax(\pi) = \#\{i \,\,\,\text{s.t.}\,\,\,\pi(i) > \pi(j) \,\,
\text{ for all }\,\,j>i, \,\,\text{ where}\,\, 1 \leq i \leq n\}.\]
We call $\rmax(\pi)$ the number of \emph{right-to-left maxima} in $\pi$.
Let
\[\ldr(\pi) = \max{\{i\,\,\,\text{s.t.}\,\,\,\pi(1)>\pi(2)>\ldots>\pi(i)\}}.\]
We call $\ldr(\pi)$ the \emph{leftmost decreasing run} of $\pi$.

\subsection{Position of first and last elements}

To prove Theorem~\ref{head.last}, we first need two propositions relating
statistics on different sets.  Let $\CT_n$ be the set of rooted plane trees on
$n$ vertices. For all $T \in \CT_n$, denote by $\delta_r(T)$ the degree of the
root vertex in~$T$. Recall that $\cd_n$ is the set of Dyck paths of length $2n$.
For all $\gamma \in \cd_n$, denote by $\alpha(\gamma)$ the number of points
on the line $y=x$ in~$\gamma$. Recall that
\[\vert \cd_n \vert = \vert \CS_n(123) \vert = \vert \CS_n(132) \vert = C_n\ts.\]
Recall that $\vert \CT_{n+1} \vert = C_n$ (see  e.g.~\cite{B1,Sta}),
so $\{\CT_n,\cd_n,\CS_n(123),\CS_n(132)\}$ are all equinumerous.

\begin{prop}\label{BijTrail}
Define the following statistics: \[ \begin{array}{lcl}
\delta_n:\CT_{n+1} \to \zz & \text{ such that }&\delta_n(T) = \delta_r(T),\\
\alpha_n:\cd_n \to \zz & \text{ such that } & \alpha_n(\gamma) = \alpha(\gamma),\\
\rmax_n:\CS_n(132) \to \zz & \text{ such that } & \rmax_n(\sigma) = \rmax(\sigma),\\
\ldr_n:\CS_n(123) \to \zz & \text{ such that } & \ldr_n(\sigma) = \ldr(\sigma),\\
\ldr'_n:\CS_n(132) \to \zz & \text{ such that }& \ldr'_n(\sigma) = \ldr(\sigma),\,\,\text{ and }\\
\first_n:\CS_n(123) \to \zz& \text{ such that }& \first_n(\sigma) = \sigma(1).\end{array} \]
Then, we have \[\delta_n \sim \alpha_n \sim \rmax_n \sim \ldr_n \sim
\ldr'_n \sim (n+1-\first_n)\quad \text{for all} \ \, n\ts.\]
\end{prop}

\smallskip

\begin{proof}[Proof of Proposition~\ref{BijTrail}]
Throughout the proof we refer to specific bijections described and analyzed
in~\cite{CK}; more details and explanation of equidistribution are available there.
We now prove equidistribution of the statistics one at a time.
\\
\\
$\bu \,\,\underline{\delta_n \sim \alpha_n}$
Recall the standard bijection $\phi:\CT_{n+1} \to \cd_n$.  Observe that
$\phi: \delta_n \to \alpha_n$. Therefore, $\delta_n \sim \alpha_n$, as desired.\\
\\
$\bu \,\,\underline{\alpha_n \sim \rmax_n}$
Let $\psi:\cd_n \to \CS_n(132)$ be the bijection $\varphi$ presented in Section~\ref{132Dyck}.
Observe that $\psi: \alpha_n \to \rmax_n$. Therefore, $\alpha_n \sim \rmax_n$, as desired.\\
\\
$\bu \,\,\underline{\rmax_n \sim \ldr_n}$
Let $\Phi$ be the West bijection between $\CS_n(132)$ and $\CS_n(123)$.
Observe that $\Phi: \rmax_n \to \ldr_n$. Therefore, $\rmax_n \sim \ldr_n$, as desired.\\
\\
$\bu \,\,\underline{\ldr_n \sim \ldr'_n}$
Let $\Psi$ be the Knuth-Richards bijection between $\CS_n(123)$ and $\CS_n(132)$.
Observe that if $\Psi(\sigma) = \tau$ for some $\sigma \in \CS_n(123)$,
then $\ldr_n(\sigma) = \ldr'_n(\tau^{-1})$. Since $\CS_n(132)$ is closed under inverses,
we get $\ldr_n(\CS_n(123)) \sim \ldr'_n(\CS_n(132))$, as desired.\\
\\
$\bu \,\,\underline{\ldr'_n \sim (n+1-\first_n)}$
Let $\Delta$ be the Knuth-Rotem bijection between $\CS_n(132)$ and $\CS_n(321)$,
and let $\Gamma:\CS_n(321) \to \CS_n(123)$ such that
$\Gamma(\sigma) = (\sigma(n),\ldots,\sigma(1))$. Observe that if
$\Gamma \circ \Delta(\sigma^{-1}) = \tau$ for some $\sigma \in \CS_n(132)$,
then $\ldr'_n(\sigma^{-1}) = n+1 - \first_n(\tau)$. Since $\sigma^{-1}$ ranges
over all $\CS_n(132)$, we get $\ldr'_n \sim \first_n$, as desired.
\qed

\medskip

The second proposition explains the behavior of $\tau(n)$, for $\tau \in \CS_n(132)$.

\begin{prop}\label{LastSymmetry}
Let $\tau \in \CS_n(132)$ be chosen uniformly at random.
Let $\.\last_n:\CS_n(132) \to \zz$ be defined as $\last_n(\tau) = \tau(n)$.
Then we have
\[\last_n \sim n+1-\last_n.\]
\end{prop}
\proof
By applying Lemma~\ref{explicitQ} for $k=n$, we obtain
\[Q_n(j,n) = C_{j-1}C_{n-j},\]
implying that
\[Q_n(j,n) = Q_n(n+1-j,n)\,\,\text{ for all integers}\,\,j\,\text{ and }\,n\,\text{ where}\,\,j\leq n.\]
Since
\[\vert \last_n^{-1}(j) \vert = Q_n(j,n)\,\,\text{ and }\,\,\vert (n+1-\last_n)^{-1}(j) \vert \.
= \. Q_n(n+1-j,n)\,\,\text{ for all }\,\,1 \leq j \leq n,\]
we have $\last_n \sim (n+1-\last_n)$, by the definition of equidistribution.
\end{proof}
\medskip

\subsection{Proof of Theorem~\ref{head.last}}
%
Let $\sigma \in \CS_n(123)$ and $\tau \in \CS_n(132)$ be chosen uniformly at random. It is known
that \[\eE[\delta_n] \to 3\,\,\text{ as }n \to \infty,\] where $\delta_n$ is defined as in
Proposition~\ref{BijTrail} (see e.g.~\cite[Example III.8]{FS}).
Since $\delta_n$ and $(n+1-\first_n)$ are equidistributed by Proposition~\ref{BijTrail},
we have \[\eE[\delta_n] = n+1-\eE[\sigma(1)].\] Therefore, we have $\eE[\sigma(1)] \to (n-2)$ as $n \to \infty$, as desired.

Due to the symmetries explained in the proof of Proposition$~\ref{Symmetry}$, we have
\[\eE[\sigma^{-1}(1)]=\eE[\sigma(1)] \to n-2 \,\,\text{ and }\,\,\eE[\sigma(n)]=\eE[\sigma^{-1}(n)] =
n+1-\eE[\sigma(1)] \to 3\,\,\text{ as }\,\,n \to \infty,\] completing the proof of $(\ref{case1})$.

Since $P_n(1,k)=b(n,k)=Q_n(1,k)$ for all $n$ and $k$ by lemmas~\ref{Pbij} and~\ref{ballot},
we have \[\eE[\tau(1)] = \eE[\sigma(1)] \to n-2\,\,\text{ as }\,\,n \to \infty.\]
Also, by Proposition~\ref{Symmetry},
we have \[\eE[\tau^{-1}(1)]=\eE[\tau(1)] \to n-2\,\,\text{ as }\,\,n \to \infty,\]
completing the proof of~$(\ref{case2})$.

To complete the proof of Theorem$~\ref{head.last}$, we need to prove~$(\ref{case3})$.
By Proposition$~\ref{LastSymmetry}$, we have $\last_n \sim (n+1-\last_n)$, so
\[\eE[\last_n] = \eE[n+1-\last_n].\]
By the linearity of expectation, we have $\eE[\last_n]=(n+1)/2$, so
\[\eE[\tau(n)] = \frac{n+1}{2},\]
as desired. By the symmetry in Proposition~\ref{Symmetry},
we have
\[\eE[\tau^{-1}(n)] = \eE[\tau(n)] = \frac{n+1}{2},\]
completing the proof.

\bigskip

\section{Fixed points in random permutations}\label{s:fp}

\subsection{Results}
Let $\sigma \in S_n$. The number of fixed points of $\sigma$ is defined as
\[\fp_n(\sigma) = \#\{i\,\,\,\text{s.t.}\,\,\,\sigma(i) = i, 1 \leq i \leq n\}.\]
In~\cite[$\S 5.6$]{E1}, Elizalde uses bijections and generating functions to obtain results on
the expected number of fixed points of permutations in $\CS_n(321),\CS_n(132)$,
and $\CS_n(123)$. In the following three theorems, the first part is due to Elizalde,
while the second parts are new results.

We use $I_{n,\ve}(a) = [(a-\ve)n,(a+\ve)n]$ to denote the intervals of elements
in $\{1,\ldots,n\}$.

\begin{thm}\label{fp321}
Let $\ve>0$, and let $\sigma \in \CS_n(321)$ be chosen uniformly at random. Then
\[\eE[\fp_n(\sigma)] = 1\,\,\text{ for all }\,\,n.\]
Moreover, for $a \in (\ve,1-\ve)$,
\[\bP\left(\sigma(i) = i\,\, \ \text{ for some }~i \in I_{n,\ve}(a)\right) \to 0\,\,\,\text{ as }\,\,n \to \infty.\]
\end{thm}

\begin{thm}\label{fp132}
Let $\ve>0$, and let $\sigma \in \CS_n(132)$ be chosen uniformly at random. Then
\[\eE[\fp_n(\sigma)] = 1,\,\,\,\text{ as }\,\,n \to \infty.\]
Moreover, for $a \in (0,1/2-\ve) \cup (1/2+\ve,1-\ve)$,
\[\bP\left(\sigma(i)=i\,\, \ \text{ for some }~i \in I_{n,\ve}(a)\right) \to 0\,\,\,\text{ as }n \to \infty.\]
\end{thm}

\begin{thm}\label{fp123}
Let $\ve>0$, and let $\sigma \in \CS_n(123)$ be chosen uniformly at random.
Then
\[\eE[\fp_n(\sigma)] \to \frac{1}{2},\,\,\,\text{ as }\,\,n \to \infty.\]
Moreover, for $a \in (0,1/2-\ve)\cup(1/2+\ve,1)$,
\[\bP\left(\sigma(i)=i\,\, \ \text{ for some }~i \in I_{n,\ve}(a)\right) \to 0\,\,\,\text{ as }n \to \infty.\]
\end{thm}

In fact, Elizalde obtains exact formulas for $\eE[\fp_n(\sigma)]$ in the last case as well~\cite[Prop.~5.3]{E1}.
We use an asymptotic approach to give independent proofs of all three theorems.

\smallskip

The final case to consider, of fixed points in \textbf{231}-avoiding permutations,
is more involved.  In the language of Section~\ref{s:big}, the expectation is equal
to the sum of entries of $\textrm{Q}_n$ along anti-diagonal~$\De$,
parallel to the wall. It is larger than in the case of the canoe since the decay
from the wall towards~$\De$ is not as sharp as in the case of the canoe.

In~\cite{E1}, Elizalde calculated the (algebraic) generating function for
the expected number of fixed points in $\CS_n(231)$, but only concluded that
$\eE[\fp_n(\sigma)]>1$ for $n\ge 3$.
Our methods allow us to calculate the asymptotic behavior of this expectation,
but not the location of fixed points.

\begin{thm}\label{fp231}
Let $\sigma \in \CS_n(231)$ be chosen uniformly at random.
Then
\[\eE[\fp_n(\sigma)] \ts \ts \sim \ts \ts\frac{2\ts\Gamma(\frac{1}{4})}{\sqrt{\pi}}\,
n^{\frac{1}{4}}\ts,\,\,\,\text{ as }\,\,n \to \infty.\]
\end{thm}

Recall that if $\sigma \in S_n$ is chosen uniformly at random, then $\eE[\fp_n(\sigma)]=1$,
so the number of fixed points statistic does not distinguish between random permutation in
$\CS_n(132)$, $\CS_n(321)$ and random permutations in~$S_n$. On the other hand, permutations
in $\CS_n(123)$ are less likely to have fixed points, in part because they can have at most~2
of them, and permutations in $\CS_n(231)$ are much more likely to have fixed points than the
typical permutation in~$S_n$.

\medskip

\subsection{Proof of Theorem~\ref{fp321}}
%
Let $\sigma \in \CS_n(321)$ be chosen uniformly at random, and
$\sigma' = (\sigma(n),\ldots,\sigma(1))$. Since $\sigma \in \CS_n(321)$, we have
$\sigma' \in \CS_n(123)$. For $\tau \in S_n$, define
$\afp(\tau) = \#\{i \,\,\text{s.t.}\,\,\tau(i) = n+1-i, 1\leq i \leq n\}$.
Let $\afp_n:\CS_n(123) \to \zz$ such that $\afp_n(\tau) = \afp(\tau)$.
Let $\fp_n:\CS_n(321) \to \zz$ such that $\fp_n(\sigma) = \fp(\sigma)$.
By definition, we have $\fp_n \sim \afp_n$.

Clearly, \[\eE[\afp_n] = \frac{1}{C_n}\sum_{k=1}^n P_n(k,n+1-k).\]
Observe that \[P_n(k,n+1-k) = C_{k-1}C_{n-k}\] by Lemma~\ref{explicitP}.
By the recurrence relation for the Catalan numbers, we have
\[\sum_{k=1}^n P_n(k,n+1-k) = C_n.\]
Therefore, \[\eE[\afp_n]=1,\,\,\,\text{ so }\,\,\eE[\fp_n] = 1.\]

For the proof of the second part, let $\ve>0$ and $a \in (\ve,1-\ve)$, and define
$\delta = \min{\{\ve,a-\ve\}}$. Let
$$\textsc{P} = \bP\left(\sigma(i) = i\,\,\text{ for some }i
\in [\delta n, (1-\delta)n+1]\right).$$
By the definition of $\textsc{P}$, we have \[\bP(\sigma(i)=i\,\,\text{ for some }i \in I_{n,\ve}(a)) \leq \textsc{P}\,\,\text{ for all }n,\] so it suffices to show that $\textsc{P} \to 0$ as $n \to \infty$.
Observe that for all~$n$, we have
\[\textsc{P} = \frac{1}{C_n}\sum_{k=\delta n}^{(1-\delta)n+1} P_n(k,n+1-k).\]
By Lemma~\ref{explicitP}, we have
\[\frac{P_n(\delta n,(1-\delta)n+1)}{C_n} = \frac{C_{\delta n-1}C_{(1-\delta)n}}{C_n} \sim
\frac{1}{4\sqrt{\pi}(\delta(1-\delta)n)^{\frac{3}{2}}},\,\,\,\text{ as }\,\,n \to \infty.\]
Observe that for $d \in (\delta,1-\delta)$ and for $n$ sufficiently large,
we have $P_n(d n,(1-d)n+1) < P_n(\delta n, (1-\delta)n+1)$.
Therefore, for $n$ sufficiently large, we have \[\textsc{P} =
\frac{1}{C_n}\sum_{k=\delta n}^{(1-\delta)n+1} P_n(k,n+1-k) \leq
\frac{(1-2\delta)n+2}{C_n} P_n(\delta n,(1-\delta)n+1) \sim
\frac{1-2\delta}{4\sqrt{\pi}(\delta(1-\delta))^{\frac{3}{2}}\sqrt{n}}.\]
Therefore, $\textsc{P} \to 0$ as $n \to \infty$, as desired.


\medskip

\subsection{Proof of theorems~\ref{fp132} and \ref{fp123}}
We omit the proofs of the probabilities tending to 0 since they are very similar
to the proof of Theorem$~\ref{fp321}$, and instead prove the following proposition.
\begin{prop}\label{fpProof} Let $\sigma\in\CS_n(123)$ and $\tau\in\CS_n(132)$
be chosen uniformly at random. Then
\[\eE[\fp_n(\tau)]=1\,\,\text{ for all }n,\,\,\text{ and }\,\,\eE[\fp_n(\sigma)]
\to \frac{1}{2},\,\,\text{ as }\,\,n \to \infty.\]
\end{prop}

\proof
Let $\sigma\in\CS_n(123)$ and $\tau\in\CS_n(132)$ be chosen uniformly at random.
By a bijection between \textbf{132}-avoiding and \textbf{321}-avoiding permutations
in~\cite{EP} (see also \cite{CK,Rob,RSZ}), fixed points are equidistributed between
$\CS_n(132)$ and $\CS_n(321)$. Therefore, by the proof of Theorem$~\ref{fp321}$,
we have $\eE[\fp_n(\tau)]=1$.

Now we prove that \[\eE[\fp_n(\sigma)] \to \frac{1}{2},\,\,\,\text{ as }\,\,n \to \infty.\]
Observe that for all $n$, we have \[\eE[\fp_n(\sigma)] = \sum_{k=1}^n \frac{P_n(k,k)}{C_n}.\]
Let $C_1$ and $C_2$ be constants such that $0<C_1<C_2$. By Theorems$~\ref{T1}$ and $\ref{T2}$,
we have \[\eE[\fp_n(\sigma)] \sim
\sum_{k=n/2-C_2\sqrt{n}}^{n/2-C_1\sqrt{n}}\frac{P_n(k,k)}{C_n} +
\sum_{k=n/2+C_1\sqrt{n}}^{n/2+C_2\sqrt{n}}\frac{P_n(k,k)}{C_n} + e(C_1,C_2,n),\]
where the error term $e(C_1,C_2,n) \to 0$ as $C_1 \to 0, C_2 \to \infty$, and $n \to \infty$.
By the symmetry in Proposition$~\ref{Symmetry}$, and by Lemma$~\ref{SmallAlpha}$, we have
\[\eE[\fp_n(\sigma)] \sim 2\sum_{k=n/2-C_2\sqrt{n}}^{n/2-C_1\sqrt{n}}\frac{P_n(k,k)}{C_n} + e(C_1,C_2,n)
\sim 2\sum_{k=C_1\sqrt{n}}^{C_2\sqrt{n}}\frac{8k^2n^{-\frac{3}{2}}}{\sqrt{\pi}}e^{-\frac{4k^2}{n}} + e(C_1,C_2,n).\]
As $C_1 \to 0$ and $C_2 \to \infty$, we have
\[\eE[\fp_n(\sigma)] \sim 2\sqrt{n}\int_0^{\infty}\frac{8c^2}{\sqrt{\pi}\sqrt{n}}e^{-4c^2}dc =
\frac{16}{\sqrt{\pi}}\int_0^{\infty}c^2e^{-4c^2}.\]
Since
\[\int_0^{\infty}c^2e^{-4c^2}dc = \frac{\sqrt{\pi}}{32},\] we get $\eE[\fp_n(\sigma)] \to 1/2$, as desired.
\qed

\medskip

\subsection{Proof of Theorem~\ref{fp231}}
\proof
Let $\sigma \in \CS_n(231)$ be chosen uniformly at random. Observe that
$\tau = (\sigma(n),\sigma(n-1),\ldots,\sigma(1)) \in \CS_n(132)$ and is distributed uniformly at random.
Therefore, \[\eE[\fp_n(\sigma)] = \eE[\afp_n(\tau)] = \sum_{k=1}^n\frac{Q_n(k,n+1-k)}{C_n}.\]
Let $A_k = Q_n(k,n-k+1)/C_n$, so $\eE[\fp_n(\sigma)] = \sum_{k=1}^nA_k$.
By Lemma~\ref{explicitQ} and Stirling's formula, we have
\[\eE[\fp_n(\sigma)] \sim 2\sum_{k=1}^{n/2}\frac{n^{\frac{3}{2}}}{\sqrt{\pi}k^{\frac{3}{2}}(n-k)^{\frac{3}{2}}}
\left(1+\sum_{r=1}^{k-1}\frac{\sqrt{r}}{\sqrt{\pi}}\exp{\left[-\frac{r^2n}{4k(n-k)}\right]}\right).\]

Let $B_K = \sum_{k=1}^{K-1}A_k, C_K = \sum_{k=K}^nA_k$, and let $K = \sqrt{n}$.
We show that $B_K = o(C_K)$. For $k<K$, we have
\[\left(\sum_{r=1}^{k-1}\frac{\sqrt{r}}{\sqrt{\pi}}\exp{\left[-\frac{r^2n}{4k(n-k)}\right]}\right)
= O(k^{\frac{3}{4}}),\,\,\text{ since }\,\,\sum_{r=1}^{C\sqrt{k}}\frac{\sqrt{r}}{\sqrt{\pi}}
\sim \frac{2Ck^{\frac{3}{4}}}{3\sqrt{\pi}}.\]
Therefore, $A_k = O(k^{-\frac{3}{4}})$, and
\[B_K = O(\sqrt{n}(\sqrt{n})^{-\frac{3}{4}}) = O(n^{\frac{1}{8}}).\]
For $k \to \infty$ as $n \to \infty$, we have
\[\left(1+\sum_{r=1}^{k-1}\frac{\sqrt{r}}{\sqrt{\pi}}\exp{\left[-\frac{r^2n}{4k(n-k)}\right]}\right)
\sim \frac{k^{\frac{3}{4}}}{\sqrt{\pi}}\int_0^{\infty}\sqrt{x}\exp{\left[-\frac{x^2n}{4k(n-k)}\right]}dx
= \frac{\sqrt{2}\,\,\Gamma\left(\frac{3}{4}\right)}{\sqrt{\pi}}\left(\frac{k(n-k)}{n}\right)^{\frac{3}{4}}.\]
Therefore, for these values of $k$, we have
\[A_k \sim \frac{\sqrt{2}\,\,\Gamma\left(\frac{3}{4}\right)}{\pi}\ts
\cdot\ts\frac{n^{\frac{3}{4}}}{k^{\frac{3}{4}}(n-k)^{\frac{3}{4}}}.\]
Consequently, we have
\[C_K = \sum_{k=\sqrt{n}}^nA_k \sim \frac{2\sqrt{2}\,\,\Gamma\left(\frac{3}{4}\right)}{\pi}\,n^{\frac{1}{4}}
\int_0^{\frac{1}{2}}(x-x^2)^{-\frac{3}{4}}dx\]\[
\sim \frac{2\sqrt{2}\,\,\Gamma\left(\frac{3}{4}\right)}{\pi}\,n^{\frac{1}{4}}
\left(\frac{2\sqrt{2\pi}\,\,\Gamma\left(\frac{5}{4}\right)}{\Gamma\left(\frac{3}{4}\right)}\right)
= \frac{2\,\,\Gamma\left(\frac{1}{4}\right)}{\sqrt{\pi}}\,n^{\frac{1}{4}}.\]
As a result, we have
\[\eE[\fp_n(\sigma)] = B_K+C_K \sim \frac{2\,\,\Gamma\left(\frac{1}{4}\right)}{\sqrt{\pi}}n^{\frac{1}{4}},\]
as desired.
\qed

\bigskip

\section{Generalized rank and the longest increasing subsequence}\label{s:lis}

\subsection{Results}
For $\la>0$, define the \emph{$\rank_{\la}$} of a permutation $\sigma \in S_n$ as the
largest integer $r$ such that $\sigma(i) > \la r$ for all $1 \leq i \leq r$.
Observe that for $\la=1$, we have $\rank_{\la} = \rank$ as defined in~\cite{EP}
(see also~\cite{CK,Kit}).

\begin{thm}\label{KRank}
Let $c_1,\la>0$ and $0 < \ve < \frac{1}{2}$. Also, let $\sigma \in \CS_n(123)$ and
$\tau \in \CS_n(132)$ be chosen uniformly at random. Let $R_{\la,n} = \rank_{\la}(\sigma)$
and $S_{\la,n} = \rank_{\la}(\tau)$. Then $\eE[R_{1,n}]=\eE[S_{1,n}]$ for all $n$.
Furthermore, there exists a constant $c_2>0$ such that for $n$ sufficiently large,
we have
\[\frac{n}{\la+1} - c_1n^{\frac{1}{2}+\ve} \leq \eE[R_{\la,n}] \leq \frac{n}{\la+1}-c_2\sqrt{n}.\]
\end{thm}
The following corollary rephrases the result in a different form.
\begin{cor}\label{logRank}
Let $\la$ and $R_n$ be as in Theorem~\ref{KRank}. Then we have
\[\lim_{n \to \infty}\frac{\log{\left(\frac{n}{\la+1}-\eE[R_n]\right)}}{\log{n}} = \frac{1}{2}.\]
\end{cor}
For $\lambda=1$, the corollary is known and follows from the following theorem
of Deutsch, Hildebrand and Wilf. Define $\lis(\sigma)$, the length of
the longest increasing subsequence in~$\sigma$, to be the largest integer~$k$
such that there exist indices $i_1<i_2<\ldots<i_k$ which satisfy
$\sigma(i_1) < \sigma(i_2)<\ldots<\sigma(i_k)$.
Let $\lis_n:\CS_n(321) \to \zz$ such that $\lis_n(\sigma) = \lis(\sigma)$.

\begin{thm}[\cite{DHW}]\label{lis}
Let $\sigma\in\CS_n(321)$. Define
\[X_n(\sigma) = \frac{\lis_n(\sigma)-\frac{n}{2}}{\sqrt{n}}\..\]
Then we have
\[\lim_{n \to \infty} \bP(X_n(\sigma) \leq \theta) =
\frac{\Gamma(\frac{3}{2},4\theta^2)}{\Gamma(\frac{3}{2})}\.,\]
where $\Gamma(x,y)$ is the \emph{incomplete Gamma function}
\[\Gamma(x,y) = \int_0^yu^{x-1}e^{-u} du.\]
\end{thm}

Theorem~\ref{lis} states that \. $\eE[\lis_n{\sigma}] \to \frac{n}{2}+c\sqrt{n}$ \,
as \. $n \to \infty$, for some constant~$c>0$. By~\cite{EP}, we have
$S_{1,n} \sim (n-\lis_n)$. By Knuth-Richards' bijection between $\CS_n(123)$ and $\CS_n(132)$,
we have $S_{1,n} \sim R_{1,n}$, so $\eE[R_{1,n}] = \eE[n-\lis_n]$.  Therefore,
$$\eE[R_{1,n}] \to n-(\frac{n}{2}+c\sqrt{n}) = \frac{n}{2}-c\sqrt{n},\,\,\text{ as }\,\,n \to \infty.$$

\medskip

\subsection{Another technical lemma} \label{ss:lis-tech}
By Knuth-Richards' bijection (also Simion-Schmidt's bijection) between
$\CS_n(123)$ and $\CS_n(132)$, the rank statistic is
equidistributed in these two classes of permutations, so $\eE[R_{1,n}] = \eE[S_{1,n}]$
for all~$n$ (see~\cite{CK,Kit}). Therefore, it suffices to prove the inequalities for~$R_{\la,n}$.

We prove the lower bound first, followed by the upper bound. To prove the lower bound,
we first need a lemma regarding this sum.
\begin{lem}\label{MaxBaseKRank}
Let $\ve>0, c_1>0, \la>0, n$ a positive integer, and $i,j$ be integers such that
\[1 \leq i \leq r , 1 \leq j \leq \la r,\,\,\text{ where }\,\,r=
\left\lfloor\frac{n}{\la+1}-c_1n^{\frac{1}{2}+\ve}\right\rfloor.\]
Then the function $P_n(i,j)$ is maximized for $(i,j) = (r,\la r)$, as $n \to \infty$.
\end{lem}

\proof[Proof of Lemma~\ref{MaxBaseKRank}]

By reasoning similar to the proof of Lemma$~\ref{BigAlpha}$,
\[\frac{P_n(r,\la r)}{C_n} \sim \frac{(\la+1)^5c_1^2n^{2\ve-\frac{1}{2}}}{4\la^{\frac{3}{2}}\sqrt{\pi}}\,
\exp{\left[-\frac{(\la+1)^4c_1^2n^{2\ve}}{4\la}\right]},\,\,\text{ as }\,\,n \to \infty.\]
Let $1 \leq i \leq r$ and $1 \leq j \leq \la r$ such that $i+j \sim sn$ for some $0 \leq s<1$.
Then by Lemma$~\ref{Neq1}$, there exists some $0<\delta<1$ such that
\[\frac{P_n(i,j)}{C_n} < \delta^n\,\,\text{ for }n\,\,\text{sufficiently large}.\]
For $n$ sufficiently large, we have
\[\delta^n < \frac{(\la+1)^5c_1^2n^{2\ve-\frac{1}{2}}}{4\la^{\frac{3}{2}}\sqrt{\pi}}\,
\exp{\left[-\frac{(\la+1)^4c_1^2n^{2\ve}}{4\la}\right]},\]
so $P_n(i,j) < P_n(r,\la r)$ as $n \to \infty$.

It remains to consider $1 \leq i \leq r, 1 \leq j \leq \la r$ such that $i+j \sim n$.
Since $r \sim n/(\la+1)$, we need
\[i \sim \frac{n}{\la+1}\,\,\text{ and }\,\,j \sim \frac{\la n}{\la+1}\]
as well. Let $i=n/(\la+1)-c$ and $j=\la n/(\la+1)-d$, where
\[c=an^{\frac{1}{2}+\ve+\alpha}\,\,\text{ and }\,\, d=bn^{\frac{1}{2}+\ve+\beta},\]
$0 \leq \alpha,\beta < 1/2-\ve$, and if $\alpha=0$ (or $\beta=0$),
then $a\geq c_1$ (or $b\geq \la c_1$, respectively).

We have
\[\frac{P_n(i,j)}{C_n} \sim \frac{(\la+1)^3(c+d+2)^2}{4\la^{\frac{3}{2}}\sqrt{\pi}n^\frac{3}{2}}
\exp{\left[-\frac{(\la+1)^2(c+d)^2}{4\la n}\right]},\]
by similar logic to that used in the proof of Lemma~\ref{BigAlpha}.
Plugging in for $c$ and $d$ in the exponent gives
\[\frac{P_n(i,j)}{C_n} \sim \frac{(\la+1)^3(c+d+2)^2}{4\la^{\frac{3}{2}}\sqrt{\pi}n^\frac{3}{2}}
\exp{\left[-\frac{(\la+1)^2(an^{\alpha}+bn^{\beta})^2n^{2\ve}}{4\la}\right]}.\]
Clearly if $\alpha>0$ or $\beta>0$ we have $P_n(i,j) < P_n(r,\la r)$ as $n \to \infty$.
Similarly, if $\alpha=\beta=0$ but $a+b > (\la+1)c_1$, then we again have
$P_n(i,j) < P_n(r,\la r)$ as $n \to \infty$. Therefore, the function $P_n(i,j)$
is indeed maximized at $(i,j)=(r,\la r)$, as desired.
\qed

\medskip

\subsection{Proof of the lower bound in Theorem~\ref{KRank}}
Let $\ve>0,\sigma \in \CS_n(123)$, and let $0<c_1$. Consider
\[\bP(R_{\la,n} \leq r),\,\,\,\text{ where}\,\,\,r = \frac{n}{\la+1}-c_1n^{\frac{1}{2}+\ve}.\]
By the union bound,
\[\bP(R_{\la,n} \leq r) \leq \sum_{i=1}^r\sum_{j=1}^{\la r} \frac{P_n(i,j)}{C_n}\..\]
By Lemma~\ref{MaxBaseKRank}, we have
\[\bP(R_{\la,n} \leq r) \leq \la r^2\,\frac{P_n(r,\la r)}{C_n}\.,\]
and by Lemma~\ref{BigAlpha}, there exists $\delta>0$ so that for $n$ sufficiently large,
we have
\[ \la r^2\,\frac{P_n(r,\la r)}{C_n}< \la r^2\delta^{n^{2\ve}} \to 0\ts.\]
Therefore,
$$
\bP\left(R_{\la,n} \ts \leq \ts \frac{n}{\la+1}-c_1\ts n^{\frac{1}{2}+\ve}\right)
\.\to\. 0 \ \ \ \text{as} \ \ n \to \infty\ts,
$$
and
$$
\eE [R_{\la,n}] \, \geq \, \frac{n}{\la+1}\. -\. c_1 \ts n^{\frac{1}{2}+\ve}\.,
$$
as desired.

\medskip

\subsection{Proof of the upper bound in Theorem~\ref{KRank}}
We can express $\eE[R_{\la,n}]$ as
\[\eE[R_{\la,n}] = \sum_{k=0}^{n/(\la+1)}k \ts \bP(R_{\la,n}=k) \ts = \ts
\frac{n}{\la+1} - \sum_{k'=0}^{n/(\la+1)}k' \ts \bP\left(R_{\la,n}=\frac{n}{\la+1}-k'\right),\]
if we let $k'=n/(\la+1)-k$.  From here, for every $0<a<b$, we have
$$
\aligned
\eE[R_{\la,n}]\, &\leq \,\frac{n}{\la+1} - \sum_{k'=a\sqrt{n}}^{b\sqrt{n}} k'\ts \bP\left(R_{\la,n}=\frac{n}{\la+1}-k'\right)
\\
& \leq \,\frac{n}{\la+1}\,-\,a\sqrt{n}\sum_{k'=a\sqrt{n}}^{b\sqrt{n}}\bP\left(R_{\la,n}=\frac{n}{\la+1}-k'\right).
\endaligned
$$
Therefore, it suffices to show that for some choice of $0<a<b$, we have
\[\bP\left(\frac{n}{\la+1}-b\sqrt{n} \leq R_{\la,n} \leq \frac{n}{\la+1}+a\sqrt{n}\right) = A>0,\]
for some constant $A = A(a,b,\la)$.

Let
\[F = \left\lfloor \frac{\la-1}{\la+1}n\right\rfloor.\]
Let $\sigma \in \CS_n(123)$, and suppose we have
\[i,j < \frac{n}{\la+1}, \sigma(i)=i+F,\,\,\text{ and }\,\,\sigma(j)=j+F.\]
Then for any $r > j$, we have $\sigma(r)<j+F$, since otherwise a
\textbf{123}-pattern would exist with $(i,j,r)$. However, this is a contradiction,
since $\sigma:\zz\cap[j+1,n] \to \zz\cap[1,j+F-1]$ must be injective, but $n-j > j+F-1$.
Consequently, $\sigma$ can have at most one value of $i < n/(\la+1)$ with $\sigma(i) = i+F$.

Let $k' = n/(\la+1)-d\sqrt{n}$ for some constant $d$. Then we have
\[\bP(R_{\la,n} \leq k') \geq \sum_{i=1}^{k'}\frac{P_n(i,i+F)}{C_n},\]
since $\sum_{i=1}^{k'}P_n(i,i+F)$ counts the number of values $i \leq k'$ such that
$\sigma(i) = i+F$ for some $\sigma \in \CS_n(123)$, and each $\sigma$ is counted by
at most one $i$. In the notation of Theorem$~\ref{T2}$, we obtain
\[\bP(R_{\la,n} \leq k') \geq \int_d^{\infty} \eta\left(\frac{1}{\la+1},t\right)
\kappa\left(\frac{1}{\la+1},t\right) dt\,\,\text{ as }\,\,n \to \infty.\]
For any $d>0$, this integral is a positive constant which is maximized at $d=a$ for $d\in [a,b]$.
Therefore,
\[\bP\left(\frac{n}{\la+1}-b\sqrt{n} \leq R_{\la,n} \leq \frac{n}{\la+1}+a\sqrt{n}\right) \to
\int_a^{b} \eta\left(\frac{1}{\la+1},t\right) \kappa\left(\frac{1}{\la+1},t\right) dt\,\,
\text{ as }\,\,n \to \infty.\]
Denote
\[A(a,b,\la) = \int_a^{b} \eta\left(\frac{1}{\la+1},t\right) \kappa\left(\frac{1}{\la+1},t\right) dt.\]

\smallskip

\noindent
For any $\lambda$ we can choose $0<a(\la)<b(\la)$ so that $A(a,b,\la)$ is bounded away from~0.
Plugging back into our upper bound gives
\[\eE[R_{\la,n}] \leq \frac{n}{\la+1}-a(\la)A(a,b,\la)\sqrt{n},\]
completing the proof of the upper bound and of Theorem~\ref{KRank}.

\bigskip

\section{Final remarks and open problems}\label{s:fin}

\subsection{} \label{ss:fin-hist}
The history of asymptotic results on Catalan numbers goes back to Euler who
noticed in 1758, that $C_{n+1}/C_n \to 4$~as $n \to \infty$, see~\cite{Eul}.
In the second half of the 20th century, the study of various statistics on
Catalan objects, became of interest first in Combinatorics and then in
Analysis of Algorithms.  Notably, binary and plane trees proved to be
especially fertile ground for both analysis and applications, and the
number of early results concentrate on these.  We refer to
\cite{A3,BPS,Dev,DFHNS,DG,FO,GW,GP,Ort,Tak} for an assortment of
both recent and classical results on the distributions of various
statistics on Catalan objects, and to~\cite{FS} for a compendium
of information on asymptotic methods in combinatorics.

The approach of looking for a limiting object whose properties can be
analyzed, is standard in the context of probability theory.  We refer
to~\cite{A1,A2} for the case of limit shapes of random trees
(see also~\cite{Drm}), and to~\cite{Ver,VK} for the early results on
limit shapes of random partitions and random Young tableaux.  Curiously,
one of the oldest bijective approach to pattern avoidance involves
infinite ``generating trees''~\cite{West}.

\subsection{} \label{ss:fin-pat}
The study of pattern avoiding permutations is very rich, and the results
we obtain here can be extended in a number of directions.  First, most
naturally, one can ask what happens to patterns of size~4,
especially to classes of equinumerous permutations not mapped into each
other by natural symmetries (see~\cite{B2}).  Of course, multiple patterns
with nice combinatorial interpretations, and other generalizations
are also of interest (see e.g.~\cite{B1,Kit}).
Perhaps, only a few of these will lead to interesting limit shapes;
we plan to return to this problem in the future.

Second, there are a number of combinatorial statistics on $\CS_n(123)$
and $\CS_n(123)$, which have been studied in the literature, and which
can be used to create a bias in the distribution.  In other words,
for every such statistic $\al:\CS_n(\pi) \to \zz$ one
can study the limit shapes of the weighted average of matrices
$$
\sum_{\si \in \CS_n(\pi)} \. q^{\al(\si)} \. M(\si)\,, \quad \, \text{where}
\ \ q\ge 0 \ \, \text{is fixed}
$$
(cf.~Subsection~\ref{s:big-setup}).  Let us single out statistic~$\al$
which counts the number of times pattern~$\om$ occurs in a permutation~$\si$.
When $\pi$ is empty, that is when the summation
above is over the whole~$S_n$, these averages interpolate
between~$S_n$ for $q=1$, and $\CS_n(\om)$ for $q\to 0$.
We refer to~\cite{B3,Hom} for closely related results
(see also~\cite{MV1,MV2})

Finally, there are natural extension of pattern avoidance to $0$-$1$ matrices,
see~\cite{KMV,Spi}, which, by the virtue of their construction, seem destined
to be studied probabilistically.  We plan to make experiments with the simple
patterns, to see if they have interesting limit shapes.

\subsection{} \label{ss:fin-geom}
There are at least nine different bijections between $\CS_n(123)$ and $\CS_n(132)$,
not counting symmetries which have been classified in the literature~\cite{CK}
(see also \cite[$\S 4$]{Kit}).  Heuristically, this suggests that none of these is
the most ``natural'' or ``canonical''. From the point of view of~\cite{P1},
the reason is that such a natural bijection would map one limit shape into
the other.  But this is unlikely, given that these limit shapes seem incompatible.

\subsection{}\label{ss:fin-int}
The integral which appears in the expression for $x(a,c)$ in Theorem~\ref{T6}
is not easily evaluated by elementary methods.  After a substitution,
it is equivalent to
\[\int_0^{\infty}\. \frac{z^2}{(z+c)^{\frac{3}{2}}} \ e^{-z^2} \ts dz,\]
which can be then computed in terms of hypergeometric and
Bessel functions\footnote{For more discussion of this integral,
see \ts {\tt http://tinyurl.com/akpu5tk}};
we refer to~\cite{AS} for definitions. Similarly, it would be nice
to find an asymptotic formula for $u(c)$ in Theorem~\ref{T4}.


\subsection{} \label{ss:fin-max}
Let us mention that the results in Section~\ref{s:main} imply few other
observations which are not immediately transparent from the figures.
First, as we mentioned in Subsection~\ref{s:big-setup}, our results imply that
the curve $Q_{n}(k,k)$ is symmetric for $(1+\ve)n/2 < k < (1-\ve)n$,
reaching the minimum at $k=3\ts n/4$, for large~$n$. Second, our results imply that the ratio
$$
\frac{Q_{n}(n/2-\sqrt{n},n/2-\sqrt{n})}{P_{n}(n/2-\sqrt{n},n/2-\sqrt{n})}
\. \to \. 2 \qquad \text{as} \ \ n \to \infty\ts,
$$
which is larger than the apparent ratios of peak heights visible in Figure~\ref{PQ250}.
Along the main diagonal, the location of the local maxima of $P_n(k,k)$
and $Q_n(k,k)$ seem to roughly coincide and have a constant ratio, as
$n\to \infty$. Our results are not strong enough to imply this, as extra
multiplicative terms can appear.  It would be interesting
to see if this is indeed the case.

\subsection{} \label{ss:fin-rank}
The generalized rank statistics \ts $\rank_{\la}$ \ts we introduce in
Section~\ref{s:lis} seem to be new.  Our numerical experiments
suggest that for all $\la>0$ and for all~$n$, \ts $\rank_{\la}$ \ts
is equidistributed between $\CS_n(132)$ and $\CS_n(123)$.
This is known for $\la=1$ (see $\S$\ref{ss:lis-tech}).
We conjecture that this is indeed the case and wonder if this follows
from a known bijection.  If true, this implies that the ``wall''
and the left side of the ``canoe'' are located at the same place
indeed, as suggested in the previous subsection.

Note that $\rank(\si)\le n/2$ for every $\si\in S_n$, since otherwise
$M(\si)$ is singular.  Using the same reasoning, we obtain
$\rank_\la(\si)\le n/(1+\la)$ for $\la\le 1$.  It would
be interesting to see if there is any connection of generalized ranks
with the longest increasing subsequences, and if these inequalities
make sense from the point of view of the Erd\H{o}s-Szekeres inequality~\cite{ES}.
We refer to~\cite{AD,BDJ} for more on the distribution of the length of 
the longest increasing subsequences in random permutations.  

\subsection{}\label{ss:fin-be}
From Lemma~\ref{explicitP}, it is easy to see that $P_n(j,k)$ for $j+k \le n+1$,
coincided with the probability that a random Dyck path of length~$2n$
passes through point $(n-j+k-1,n+j-k-1)$.  This translates the problem of
computing the limit shape of the ``canoe'' to the shape of Brownian excursion,
which is extremely well understood (see~\cite{Pit} and references therein).
As mentioned in the introduction, this explains all qualitative phenomena
in this case.  For example, the expected maximum distance from the anti-diagonal
is known to be $\sqrt{\pi\ts n}\ts (1+o(1))$ (see~\cite{Chu,DI}).  Similarly,
the exponential decay of $P_n(k-t,n-k-t)$ for $t=n^{1/2+\ve}$,
follows from the setting, and seems to correspond to tail estimates
for the expected maximal distance.  However, because of the emphasis on
the maxima and occupation time of Brownian excursions, it seems there are no
known probabilistic analogues for results such as our Theorem~\ref{T2} despite
a similarities of some formulas. For example, it is curious that for
$c\ne 0$ and $\al=1/2$, the expression
$$\eta(a,c) \ts \kappa(a,c) \, = \,
\frac{c^2}{\sqrt{\pi}\. a^{\frac{3}{2}} \ts (1-a)^{\frac{3}{2}}} \. \exp{\left[{\frac{-c^2}{a(1-a)}}\right]}
$$
is exactly the density function of a Maxwell-distributed random variable,
which appears in the contour process of the Brownian excursion (cf.~\cite{GP}).

\subsection{}\label{ss:fin-proc}
Unfortunately, there seem to be no obvious way to interpret Lemma~\ref{explicitQ}
probabilistically.  One can of course, use bijections to random binary trees,
but the corresponding statistics are not very natural.  It would be interesting
to find a good probabilistic model with the same limit shape as $\CS_n(132)$.\footnote{Most
recently, such model was found by Christopher Hoffman, Erik Slivken and Doug Rizzolo, using
the \emph{tunnel} concept from~\cite{E1} (in preparation).}

\subsection{} \label{ss:fin-chi}
Let us define the following variation on the $\chi$-squared statistic on~$S_n$~:
$$
\chi^2(\sigma) \. := \. \sum_{i=1}^n \. r(i)\ts, \quad \text{where} \ \
r(i) = \min\bigl\{(n+1-\sigma(i)-i)^2, (2n-\sigma(i)-i)^2\bigr\}.
$$
This statistic measures how far the permutation~$\si$ is from the reverse
identity permutation (in cyclic order).  Curiously, in contrast with the
number of fixed points, this statistic can distinguish our sets of pattern
avoiding permutations.

\begin{thm} \label{t:chi}
For $\chi^2$ defined as above and $n\to \infty$, we have:
$$\aligned
\eE[\chi^2(\si)] \. & = \. \Theta(n^2)\ts, \quad \ \, \text{where} \ \ \si\in \CS_n(123) \ \, \text{uniform}\ts,
\\
\eE[\chi^2(\si)] \. & = \. \Theta(n^{2.5})\ts, \quad \text{where} \ \ \si\in \CS_n(132) \ \, \text{uniform}\ts,
\\
\eE[\chi^2(\si)] \. & = \. \Theta(n^3)\ts, \quad \. \ \, \text{where} \ \, \ \si\in S_n \ \, \text{uniform}\ts.
\endaligned
$$
\end{thm}

\noindent
The remaining four patterns of length~$3$ have also these asymptotics by the symmetries.
The proof of Theorem~\ref{t:chi} will appear in a forthcoming thesis~\cite{Min}
by the first author.

\subsection{}
After this paper was written and posted on the {\tt arXiv}, we learned
of two closely related papers.  In~\cite{ML}, the authors set up a
related random pattern avoiding permutation model and make a number
of Monte Carlo simulations and conjectures, including suggesting an
empiric ``canoe style'' shape.  Rather curiously, the authors prove
the exponential decay of the probability $P\bigl(\tau(1)>0.71\ts n\bigr)$,
for random $\tau \in\CS_n(4231)$.

In~\cite{AM}, the authors prove similar ``small scale'' results for
patterns of size~3, i.e.~exponential decay above anti-diagonal
and polynomial decay below anti-diagonal for random $\si \in\CS_n(132)$.
They also study a statistic similar but not equal to~$\rank$.
The first author surveys these results and explores the
connections in~\cite{Min}.

\vskip.5cm

\noindent
\textbf{Acknowledgments:} \. The authors are grateful to
Ton{\'c}i Antunovi{\'c}, Drew Armstrong, Marek Biskup,
Stephen DeSalvo, Sergi Elizalde, Sergey Kitaev, Jim Pitman and 
Richard Stanley, for useful remarks and help with the references.
Special thanks to Neal Madras for telling us about 
his ongoing work with Lerna Pehlivan, and hosting the first
author during his Summer 2013 visit.  We are indebted to an anonymous 
referee for careful reading of the paper, comments and helpful 
suggestions.  The second author was partially supported by the
BSF and the~NSF.

 \vskip1.6cm




\newpage

\section{Appendix: Numerical calculations}

\vskip.3cm

\begin{figure}[ht!]
	\centering
	\begin{minipage}{.5\textwidth}
		\centering
		\includegraphics[width=60mm]{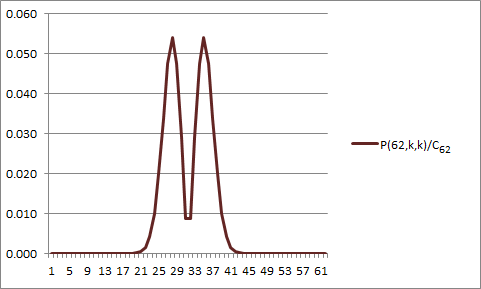}
		\captionof{figure}{Values of $P_{62}(k,k)$.}
		\label{Qkk62}
	\end{minipage}%
	\begin{minipage}{.5\textwidth}
		\centering
		\includegraphics[width=60mm]{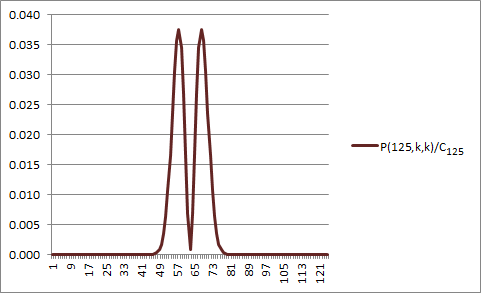}
		\captionof{figure}{Values of $P_{125}(k,k)$.}
		\label{Q125kk}
	\end{minipage}

\vskip.48cm

	\begin{minipage}{.5\textwidth}
		\centering
		\includegraphics[width=60mm]{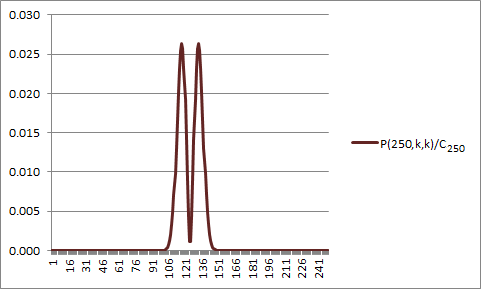}
		\captionof{figure}{Values of $P_{250}(k,k)$.}
		\label{Qkk250}
	\end{minipage}%
	\begin{minipage}{.5\textwidth}
		\centering
		\includegraphics[width=60mm]{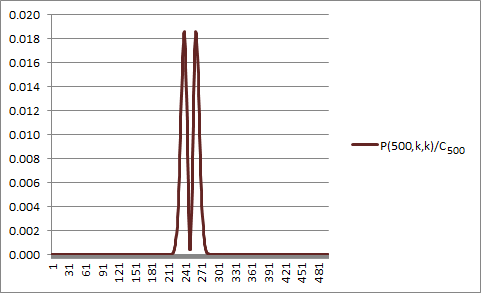}
		\captionof{figure}{Values of $P_{500}(k,k)$.}
		\label{Q500kk}
	\end{minipage}

\end{figure}

\vskip.2cm

\begin{figure}[ht!]
	\centering
	\begin{minipage}{.5\textwidth}
		\centering
		\includegraphics[width=60mm]{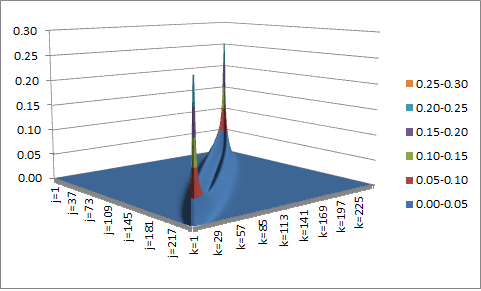}
		\captionof{figure}{\small Surface of $P_{250}(j,k)$.}
		\label{P250}
	\end{minipage}%
	\begin{minipage}{.5\textwidth}
		\centering
		\includegraphics[width=60mm]{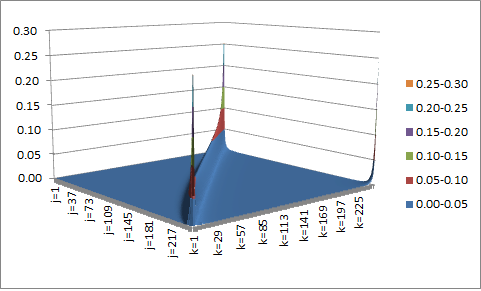}
		\captionof{figure}{\small Surface of $Q_{250}(j,k)$.}
		\label{Q250}
	\end{minipage}

\vskip.48cm

	\begin{minipage}{.5\textwidth}
		\centering
		\includegraphics[width=60mm]{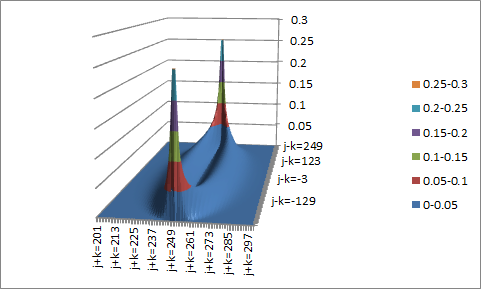}
		\captionof{figure}{\small A closer look at $P_{250}(j,k)$, $201 \leq j+k \leq 301$.}
		\label{P250close}
	\end{minipage}%
	\begin{minipage}{.5\textwidth}
		\centering
		\includegraphics[width=60mm]{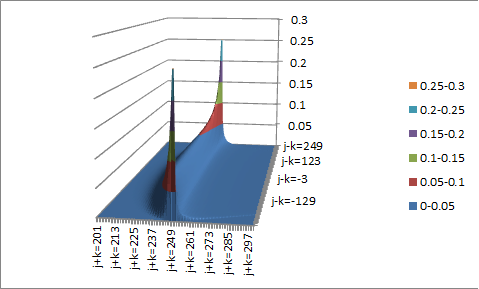}
		\captionof{figure}{{\small A closer look at $Q_{250}(j,k)$, $201 \leq j+k \leq 301$.}}
        \label{Q250close}
	\end{minipage}
\end{figure}


\begin{thebibliography}{DFHNS}

\bibitem[AS]{AS}
M.~Abramowitz and I.~Stegun, \emph{Handbook of mathematical functions},
Dover, New York, 1965.

\bibitem[A1]{A1}
D.~Aldous, The continuum random tree. I, \emph{Ann.~Probab.}~\textbf{19} (1991), 1--28.

\bibitem[A2]{A2}
D.~Aldous, Recursive self-similarity for random trees, random triangulations,
and Brownian excursion, \emph{Ann.~Probab.}~\textbf{22} (1994), 527--545.

\bibitem[A3]{A3}
D.~Aldous, Triangulating the circle, at random,
\emph{Amer.~Math.~Monthly}~\textbf{101} (1994), 223--233.

\bibitem[AD]{AD}
D.~Aldous and P.~Diaconis, Longest increasing subsequences: from patience sorting to
the Baik-Deift-Johansson theorem, \emph{Bull.~AMS}~\textbf{36} (1999), 413--432.

\bibitem[AM]{AM}
M.~Atapour and N.~Madras, Large Deviations and Ratio Limit Theorems
for Pattern-Avoiding Permutations, unpublished preprint (2013).


\bibitem[BDJ]{BDJ}
J.~Baik, P.~Deift and K.~Johansson, On the distribution of the length of the
longest increasing subsequence of random permutations,
\emph{J.~AMS}~\textbf{12} (1999), 1119--1178.

\bibitem[BPS]{BPS}
N.~Bernasconi, K.~Panagiotou and A.~Steger,
On properties of random dissections and triangulations,
\emph{Combinatorica}~\textbf{30} (2010), 627--654.

\bibitem[B1]{B2}
M.~B\'{o}na, Permutations avoiding certain patterns: the case of length~4
and some generalizations, \emph{Discrete Math.} \textbf{175} (1997), 55--67.

\bibitem[B2]{B1}
M.~B\'{o}na, \emph{Combinatorics of Permutations}, CRC Press, Boca Raton, FL, 2012.

\bibitem[B3]{B3}
M.~B\'{o}na, Surprising symmetries in objects counted by Catalan numbers,
\emph{Electron.~J. Combin.}~\textbf{19} (2012), no.~1, Paper 62, 11~pp.

\bibitem[Chu]{Chu}
K.~L.~Chung,  Excursions in Brownian motion,
\emph{Ark. Mat.}~\textbf{14} (1976), 155--177.

\bibitem[CK]{CK}
A.~Claesson and S.~Kitaev, Classification of bijections between 321-and
132-avoiding permutations, \emph{S\'{e}m.~Loth.~Comb.}~\textbf{60} (2008), Art.~B60d.

\bibitem[DHW]{DHW}
E.~Deutsch, A.~Hildebrand and H.~Wilf,
Longest increasing subsequences in pattern-restricted permutations,
\emph{Electron.~J.~Combin.}~\textbf{9} (2002), R~12, 8~pp.

\bibitem[Dev]{Dev}
L.~Devroye, Branching processes in the analysis of the heights of trees,
\emph{Acta Inform.}~\textbf{24} (1987), 277--298.

\bibitem[DFHNS]{DFHNS}
L.~Devroye, P.~Flajolet, F.~Hurtado, M.~Noy and W.~Steiger,
Properties of random triangulations and trees,
\emph{Discrete Comput.~Geom.}~\textbf{22} (1999), 105--117.

\bibitem[Drm]{Drm}
M.~Drmota, \emph{Random Trees}, Springer, Vienna, 2009.

\bibitem[DG]{DG}
M.~Drmota and B.~Gittenberger, On the profile of random trees,
\emph{Random Struct.~Alg.}~\textbf{10} (1997), 421--451.

\bibitem[DI]{DI}
R.~T.~Durrett and D.~L.~Iglehart,
Functionals of Brownian meander and Brownian excursion,
\emph{Ann. Probab.}~\textbf{5} (1977), 130-–135.

\bibitem[E1]{E1}
S.~Elizalde, \emph{Statistics on pattern-avoiding permutations}, Ph.D.~Thesis,
MIT, 2004, 116~pp.

\bibitem[E2]{E2}
S.~Elizalde, Fixed points and excedances in restricted permutations,
\emph{Electron.~J.~Combin.}~\textbf{18} (2011), Paper 29, 17~pp.

\bibitem[EP]{EP}
S.~Elizalde and I.~Pak, Bijections for refined restricted permutations,
\emph{J.~Combin.~Theory, Ser.~A}~\textbf{105} (2004), 207--219.

\bibitem[ES]{ES}
P.~Erd\H{o}s and Gy.~Szekeres, A combinatorial problem in geometry,
\emph{Compositio Math.}~\textbf{2} (1935), 463--470.

\bibitem[Eul]{Eul}
L.~Euler, Summary of ``Enumeratio modorum quibus figurae planae..., by I.~A.~de~Segner'',
\emph{Novi Commentarii Academiae Scientiarum Imperialis Petropolitanae}~\textbf{7}
(1758/59), 13--15; English translation available at \ts {\tt http://tinyurl.com/bhvsht2}

\bibitem[FO]{FO}
P.~Flajolet and A.~Odlyzko, The average height of binary trees and other simple trees,
\emph{J.~Comput. System~Sci.}~\textbf{25} (1982), 171--213.

\bibitem[FS]{FS}
P.~Flajolet and R.~Sedgewick, \emph{Analytic Combinatorics},
Cambridge Univ.~Press, Cambridge, 2009.


\bibitem[GW]{GW}
Z.~Gao and N.~C.~Wormald, The distribution of the maximum vertex degree in random planar maps,
\emph{J.~Combin.~Theory, Ser.~A} \textbf{89} (2000), 201--230.

\bibitem[Gar]{G}
M.~Gardner, Mathematical Games, Catalan numbers: an integer sequence that materializes
in unexpected places, \emph{Scientific Amer.}~\textbf{234}, No.~6 (June, 1976), 120--125,~132.

\bibitem[Gou]{Gould}
H.~W.~Gould, \emph{Bell and Catalan Numbers: A Research Bibliography of Two Special
Number Sequences} (Revised~2007);
available at \ts {\tt http://www.math.wvu.edu/\/\~\/gould/BibGood\%20final.pdf}

\bibitem[GP]{GP}
W.~Gutjahr and G.~Pflug, The asymptotic distribution of leaf heights in binary 
trees, \emph{Graphs and Combin.}~\textbf{8} (1992), 243--251.

\bibitem[Hom]{Hom}
C.~Homberger,
Expected patterns in permutation classes,
\emph{Electron.~J. Combin.}~\textbf{19} (2012), no.~3,
Paper~43, 12~pp.

\bibitem[Kim]{Kim}
H.~S.~Kim, Interview with Professor Stanley, \emph{MIT~Math Majors Magazine}~\textbf{1},
No.~1 (2008), 28--33; available at \ts {\tt http://web.mit.edu/uma/www/mmm/mmm0101.pdf}

\bibitem[Kit]{Kit}
S.~Kitaev, \emph{Patterns in Permutations and Words}, Springer, Berlin, 2011.

\bibitem[KMV]{KMV}
S.~Kitaev, T.~Mansour and A.~Vella,
Pattern avoidance in matrices, \emph{J.~Integer Seq.}~\textbf{8} (2005), no.~2,
Art.~05.2.2, 16~pp.

\bibitem[Knu]{Knu}
D.~Knuth, \emph{The Art of Computer Programming}, vol.~1, Addison-Wesley, Reading, MA, 1968.

\bibitem[Kra]{Kra}
C.~Krattenthaler, Permutations with restricted patterns and Dyck paths,
\emph{Adv. Appl. Math.}~\textbf{27} (2001), 510--530.

\bibitem[Mac]{M}
P.~A.~MacMahon, \emph{Combinatory Analysis}, vol.~1,
Cambridge Univ.~Press, Cambridge, 1915.

\bibitem[ML]{ML}
N.~Madras and H.~Liu,
Random Pattern-Avoiding Permutations, in
\emph{Algorithmic Probability and Combinatorics},
AMS, Providence, RI, 2010, 173--194.

\bibitem[MV1]{MV1}
T.~Mansour and A.~Vainshtein, Restricted 132-avoiding permutations,
\emph{Adv.~Appl.~Math}~\textbf{26} (2001), 258--269.

\bibitem[MV2]{MV2}
T.~Mansour and A.~Vainshtein, Counting occurences of 132 in a permutation,
\emph{Adv.~Appl.~Math}~\textbf{28} (2002), 185--195.

\bibitem[Min]{Min}
S.~Miner, \emph{Limiting shapes of combinatorial objects}, Ph.D thesis, University of California, Los Angeles, in preparation.

\bibitem[Ort]{Ort}
J.~Ortmann, Large deviations for non-crossing partitions,
\emph{Electron.~J.~Probab.}~\textbf{17} (2012), 1--25.

\bibitem[P1]{P1}
I.~Pak, The nature of partition bijections~II. Asymptotic stability,
preprint, 32~pp.; available at \ts
{\tt http://www.math.ucla.edu/\/\~\/pak/papers/stab5.pdf}

\bibitem[P2]{Pak}
I.~Pak, Catalan Numbers Page, \ts
{\tt http://www.math.ucla.edu/\/\~\/pak/lectures/Cat/pakcat.htm}

\bibitem[Pit]{Pit}
J.~Pitman, Combinatorial stochastic processes,
\emph{Lecture Notes in Mathematics}~\textbf{1875}, Springer,
Berlin, 2006, 256~pp.


\bibitem[Rob]{Rob}
A.~Robertson, Restricted permutations from Catalan to Fine and back,
\emph{S\'{e}m.~Loth.~Comb.}~\textbf{50} (2004), Art.~B50g.

\bibitem[Rus]{Rus}
F.~Ruskey, On the average shape of binary trees, \emph{SIAM J.~Algebraic Discrete Methods}~\textbf{1} (1980), 43--50.

\bibitem[RSZ]{RSZ}
A.~Robertson, D.~Saracino and D.~Zeilberger, Refined restricted permutations,
\emph{Ann.~Comb.}~\textbf{6} (2002), 427--444.

\bibitem[Rota]{Rota}
G.-C.~Rota, Ten lessons I wish I had been taught,
\emph{Notices AMS}~\textbf{44} (1997), 22--25.


\bibitem[Spi]{Spi}
A.~Spiridonov, \emph{Pattern-avoidance in binary fillings of grid shapes},
Ph.D.~Thesis, MIT, 2009, 88~pp.

\bibitem[Slo]{Slo}
N.~J.~A.~Sloane, A000108, in \emph{Online Encyclopedia of Integer Sequences},
\ts {\tt http://oeis.org/A000108}


\bibitem[S1]{Sta}
R.~Stanley, \emph{Enumerative Combinatorics}, vols.~1 and 2,
Cambridge Univ.~Press, Cambridge, 2011.

\bibitem[S2]{S2}
R.~Stanley, Catalan addendum, available at \ts {\tt http://www-math.mit.edu/\/\~\/rstan/ec/catadd.pdf}

\bibitem[Tak]{Tak}
L.~Tak\'{a}cs, A Bernoulli excursion and its various applications,
\emph{Adv. Appl. Probab.}~\textbf{23} (1991), 557--585.

\bibitem[Ver]{Ver}
A.~M.~Vershik,
Statistical mechanics of combinatorial partitions, and their limit configurations,
\emph{Funct. Anal. Appl.}~\textbf{30} (1996), 90--105.

\bibitem[VK]{VK}
A.~Vershik and S.~Kerov, Asymptotics of the Plancherel measure of the symmetric group and the
limiting form of Young tables, \emph{Soviet Math. Dokl.}~\textbf{18} (1977), 527--531.

\bibitem[West]{West}
J.~West, Generating trees and the Catalan and Schr\"{o}der numbers,
\emph{Discrete Math.}~\textbf{146} (1995), 247--262.

\end{thebibliography}
\end{document}